\documentclass[a4paper,oneside,10.5pt]{article}%
\usepackage{amsmath}
\usepackage{amsfonts}
\usepackage{indentfirst}
\usepackage{amssymb}
\usepackage{graphicx}%
\usepackage{appendix}

\usepackage[colorlinks,citecolor=blue,urlcolor=blue]{hyperref}

\providecommand{\U}[1]{\protect \rule{.1in}{.1in}}
\newtheorem{theorem}{Theorem}[section]

\newtheorem{example}[theorem]{Example}

\newtheorem{lemma}[theorem]{Lemma}

\newtheorem{proposition}[theorem]{Proposition}
\newtheorem{remark}[theorem]{Remark}

\newenvironment{proof}[1][Proof]{\noindent \textbf{#1.} }{\  \rule{0.5em}{0.5em}}
\numberwithin{equation}{section}

\begin{document}
	
	\title{Infinite Anticipation Backward Stochastic Differential Equations}
	\author{Guanwei Cheng\thanks{School of Mathematics, Shandong University, PR China, (guanwei.cheng@mail.sdu.edu.cn).}\quad Shuzhen Yang\thanks{Shandong University-Zhong Tai Securities Institute for Financial Studies, Shandong University, PR China, (yangsz@sdu.edu.cn).}
		\thanks{This work was supported by National Key R\&D program of China (Grant No.2023YFA1009203), National Natural Science Foundation of China (Grant No.12471450), and Taishan Scholar Talent Project Youth Project.}}
	\date{July 2025}
	\maketitle
	
	\textbf{Abstract}:  In this paper, we introduce a new type of backward stochastic differential equations (BSDEs) with infinite anticipation, where the generator depends on the entire future values of the solution in infinite horizon. We show that the new BSDEs has a unique solution and admits a comparison result. In the end, we solve a stochastic control problem via a duality between BSDEs with infinite anticipation and stochastic differential equations (SDEs) with infinite delay.\\
	\textbf{Keywords}:
	{\textbf{MSC2010}: Anticipated BSDEs; Infinite anticipation; Duality
		
		\addcontentsline{toc}{section}{\hspace*{1.8em}Abstract}
		
		\section{Introduction}
		In the 1970s,  linear Backward stochastic differential equations (BSDEs) was first introduced by Bismut \cite{bismut} for optimal control problems, and have since become a cornerstone of mathematical finance and stochastic analysis. The fundamental advancement in non-linear BSDEs theory emerged in Peng and Pardoux \cite{peng90}, where the generator depends on current values of the solutions. However, the practical problems always involve delays or anticipations in decision-making which motivate the following work. Peng and Yang \cite{peng09AP} introduced a new type of BSDEs with finite anticipation:
		\begin{equation}\label{peng09equation}
			\left\{\begin{aligned}
				Y_t & =\xi_T + \int_{t}^{T}f\left(s, Y_s, Z_s, Y_{s+\delta(s)}, Z_{s+\zeta(s)}\right) d s-\int_{t}^{T}Z_s d W_s, & & t \in[0, T] ; \\
				Y_t & =\xi_t, & & t \in[T, T+K] ; \\
				Z_t & =\eta_t, & & t \in[T, T+K],
			\end{aligned}\right.
		\end{equation}
		%This innovation bridged BSDEs with stochastic differential delay equations (SDDEs) and revealed a novel duality between delayed forward equations and anticipated backward ones
		where generator explicitly depends on future values of the solution $\left(Y_{s+\delta(s)}, Z_{s+\zeta(s)}\right)$. Based on Lipschitz conditions of $f$ on variables, Peng and Yang \cite{peng09AP} proved
		that equation (\ref{peng09equation}) has a unique solution and established the related comparison results. Furthermore, they resolved an optimal control problem via a novel duality between stochastic differential delay equations (SDDEs) and anticipated BSDEs.
		
		Subsequent research extended the framework of Peng and Yang \cite{peng09AP} in many directions.
		Some works studied anticipated BSDEs driven by Poisson jumps (\cite{drivenPoissonjumps2,drivenPoissonjumps}), fractional Brownian motion (\cite{driven2fractionalBM, drivenfractionalBM2}) and other processes (\cite{drivencontinous,drivensinglejumps,drivenmarkovchain}),
		while other literature focus on non-Lipschitz (\cite{leftlip,nonlip1}) or quadratic growth
		(\cite{quadraticgrowthwithjumps,quadraticgrowth1}) coefficients.
		For more types of anticipated BSDEs, we refer to doubly (\cite{2ABSDEwithnonlip,2ABSDE2}), reflected (\cite{withreflection2,withreflection1}), mean-field (\cite{MFdrivenfractionalBM,MFdrivenLevynoises}) and Volterra integral (\cite{Volterraintegralwithjump,Volterraintegral2}) anticipated BSDEs.
		
		{Moreover, several studies examined finitely-anticipated BSDEs in infinite-time horizons as the adjoint equation in the stochastic control problems (see \cite{infinitehorizon4,infinitehorizon1,infinitehorizon3}). Since their generators remain constrained to finite anticipation windows, typically depending on future values within a bounded interval $[s, s+\theta]$, we thus classify these equations as "BSDEs with finite anticipation".}	A parallel framework of anticipated BSDEs, called "generalized anticipated BSDEs", was introduced in Yang \cite{yangzhe07GABSDE}, and was later generalized in Yang and Elliot \cite{yangzhe13ECP}, as follows:
		$$
		\left\{\begin{aligned}
			Y_t=&\xi_T + \int_t^{T} f\left(s, \{Y_r\}_{r \in [s,T+K]}, \{Z_r\}_{r \in [s,T+K]} \right) d s- \int_t^{T} Z_s d W_s, \quad &&t \in[0, T]; \\
			Y_t = & \xi_t, \quad && t\in[T, T+K]; \\
			Z_t = & \eta_t, \quad && t\in[T, T+K],
		\end{aligned}\right.
		$$
		%\textcolor{red}{?a???¨°?¨¹??¦Ì?o¨®D?1¡è¡Á¡Â?¡Â¨°a?¡¥?D?¨²?a??¨°?D?¨¬??¡§D?¨º?¦Ì????¨²delayed?????¨º¨¬a?D¡ê?¨¢D???a?????¡ê}
		where the generator depends on the entire future value of the solutions, within bounded horizons $[s,T+K]$.
		
		In this paper, we consider the extension of "generalized anticipated BSDEs" to the \textbf{ infinite anticipation BSDEs (IABSDEs)}:
		$$
		\left\{\begin{aligned}
			Y_t=&\xi_T + \int_t^{T} f\left(s, \{Y_r\}_{r \in [s,+\infty)}, \{Z_r\}_{r \in [s,+\infty)} \right) d s- \int_t^{T} Z_s d W_s, \quad &&t \in[0, T]; \\
			Y_t = & \xi_t, \quad && t\in[T, +\infty); \\
			Z_t = & \eta_t, \quad && t\in[T, +\infty).
		\end{aligned}\right.
		$$
		Here, the generator depends on unbounded future trajectory of the solutions, enabling full anticipation over $[s,+\infty)$. Under mild assumptions on $f$, we establish the well-posedness for IABSDEs and based on which solves a stochastic optimal control problem.
		
		The main contributions of this paper are threefold. First, we establish the well-posedness of IABSDEs under a weaker Lipschitz condition (see Section \ref{chapter condition compare}) that accommodates dependencies on whole unbounded future trajectory, generalizing prior frameworks \cite{yangzhe13ECP}. Second, we extend  comparison results in finite-horizon (see \cite{peng09AP,yangzhe13ECP}) to the infinite anticipation setting without additional technical constraints. Finally, we establish a duality between IABSDEs and SDEs with infinite delay (ISDDEs), a class of memory processes with diverse physical and economic applications such as species growth or incubating time on disease models (see \cite{ISDDEapplication1, ISDDEapplication2}). This duality provides a novel tool to investigate more types of control problems governed by ISDDEs, as empirically demonstrated in climate-economic systems where permanent damages depend on historical cumulative carbon emission (see Section \ref{chapter stochastic control problem}).
		
		The remainder of this paper is organized as follows. Section \ref{chapter preliminary} presents some classical results of BSDEs. In Section \ref{chapter formulation}, we introduce the formulation of IABSDEs, and examine the priori estimate of the solutions. In Section \ref{chapter uniqueness and existence}, we prove the existence and uniqueness theorem for the IABSDEs, and establish the comparison results in Section \ref{chapter comparison}. Finally in Section \ref{chapter Stochastic control problem}, based on the new IABSDEs, we solve a stochastic control problem by a duality property.
		
		\section{Preliminary}\label{chapter preliminary}
		\indent
		Let $\left(\Omega, \mathcal{F}, P, \{\mathcal{F}_t\}_{t \geqslant 0}\right)$ be a complete probability space such that $\mathcal{F}_0$ contains all $P$ null elements of $\mathcal{F}$ and suppose that the filtration is generated by a $d$-dimensional standard Brownian motion $W=\left(W_t\right)_{t \geqslant 0}$. Let $T>0$, and for all $n \in \mathbb{N}$, Euclidean norm in $\mathbb{R}^d$ is denoted by $|\cdot|$.
		
		%Let $B C\left([0, +\infty) ; \mathbb{R}^d\right)$ denote the family of bounded continuous $\mathbb{R}^d$-value functions $\varphi$ defined on $[0, +\infty)$ with norm $\|\varphi\|=\sup _{0 \leqslant \theta < +\infty}|\varphi(\theta)|$. It is apparent that $B C\left([0, +\infty) ; \mathbb{R}^d\right)$ is a Banach space.
		We introduce the following notations for stochastic processes in finite time:
		\begin{itemize}
			\item $L^2\left(\mathcal{F}_t ; \mathbb{R}^d\right)=\left\{\mathbb{R}^d\right.$-valued $\mathcal{F}_t$-measurable random variables such that $\left.E\left[|\xi|^2\right]<\infty\right\}$;
			\item $\mathcal{M}_{\mathcal{F}}^2\left(0, T ; \mathbb{R}^d\right)=\left\{\mathbb{R}^d\right.$-valued and $\mathcal{F}_t$-adapted stochastic processes such that\\ $\left.E\left[\int_0^T\left|\varphi_t\right|^2 d t\right]<\infty\right\} ;$
			\item $\mathcal{S}_{\mathcal{F}}^2\left(0, T; \mathbb{R}^d\right)=\left\{\right.$continuous stochastic processes in $\mathcal{M}_{\mathcal{F}}^2\left(0, T ; \mathbb{R}^d\right)$ such that\\ $\left.E\left[\sup _{0 \leqslant t \leqslant T}\left|\varphi_t\right|^2\right]<\infty\right\}.$
		\end{itemize}
		
		For the stochastic processes with infinite time, we denote, for some given constant $\beta$: %\\\textcolor{red}{1.$\mathcal{S}_{\mathcal{F}}^2$????¦Ì?phase path??o????¨¹¨º???o?¡ä|¡ä|¨°a¨°???¨®D??¡ê?2??a¦Ì¨¤¨º?¡¤??¨¢¨®D¨®¡ã?¨¬¡ê¡§??¨°a¨ª¨º¨¨?¨®D?????¨¹??¨°?2?¨°a?¨®¨¦¨²3¨¦?a¨°???¨®D??¡ê?¨¦??¨¢?¨¦¨°??¡À?¨®????¡¤??¨²phase space ¨º?BC¦Ì??¨°?¨¹¨°a?¨® ¡ê?¨°¨°?a??¨¨?????o???¨®D¡Á¡Â¨®?¡ê?¡ê?¡ê?\\2.¨¨?1?¨¨£¤¦Ì?BC¦Ì??T??¡ê?¨¨???¡ã?Y¨®D1?¦Ì?¨¨?2?¡ê¡§¡ã¨¹¨¤¡§????¨¬??t¡ê?¦Ì??????¨ª¡À?3¨¦$\mathcal{S}_{\mathcal{F}}^2$¡ê??¨¹??¨°?2?¦Ì?¡ê???3¨¦o¨ªZ¨°??¨´¦Ì?¡ä?$\beta$¦Ì??a$\mathcal{S}_{\mathcal{F}}^{2,\beta}$¨°2¨º??¨¦¨°?3¨¦¨¢¡é¦Ì?}
		\begin{itemize}
			\item $\mathcal{M}_{\mathcal{F}}^{2,\beta}\left(0, +\infty ; \mathbb{R}^d\right)=\left\{\mathbb{R}^d\right.$-valued and $\mathcal{F}_t$-adapted stochastic processes such that\\ $\left.E\left[\int_0^{+\infty}e^{\beta t}\left|\varphi_t\right|^2 d t\right]<\infty\right\} ;$
			\item $\mathcal{S}_{\mathcal{F}}^2\left(0, +\infty ; \mathbb{R}^d\right)=\left\{\right.$continuous stochastic processes in $\mathcal{M}_{\mathcal{F}}^2\left(0, +\infty ; \mathbb{R}^d\right)$ such that\\ $\left.E\left[\sup _{0 \leqslant t \textless +\infty}\left|\varphi_t\right|^2\right]<\infty \right\}.$
			%with $\left. B C\left([0, +\infty) ; \mathbb{R}^d\right)\right.$being the phase space$\left.\right\}$.
		\end{itemize}
		
		Note that all the spaces mentioned above are Banach spaces. If $d=1$, we further denote $L^2\left(\mathcal{F}_t; \mathbb{R}\right)$, $\mathcal{M}_{\mathcal{F}}^2(0, T ; \mathbb{R})$, $\mathcal{S}_{\mathcal{F}}^2(0, T ; \mathbb{R})$, $\mathcal{S}_{\mathcal{F}}^2(0, +\infty ; \mathbb{R})$, by $L^2\left(\mathcal{F}_t\right)$, $\mathcal{M}_{\mathcal{F}}^2(0, T)$, $\mathcal{S}_{\mathcal{F}}^2(0, T)$, $\mathcal{S}_{\mathcal{F}}^2(0, +\infty)$, respectively.

		We now review some classical results on BSDEs that are used in this study. Lemma \ref{classic BSDE uniqueness and exsitence} is an existence and uniqueness result for BSDEs that is given in Theorem 3.2 of Peng \cite{Peng2004}. Lemma \ref{classic strict BSDE comparison} is given in Theorem 3.3 of Peng \cite{Peng2004}, which is a comparison result for solutions of BSDEs, and also see  El Karoui, Peng and Quenez \cite{BSDEinfinance}. Lemma \ref{classic BSDE comparison} is a direct corollary of Lemma \ref{classic strict BSDE comparison}.
		
		We consider the following conditions for $g=g(\omega, t, y, z): \Omega \times[0, T] \times \mathbb{R}^d \times \mathbb{R}^{d \times m} \rightarrow \mathbb{R}^d$ :
		
		\textbf{(a)} $g(\cdot, y, z)$ is an $\mathbb{R}^d$-valued and $\mathcal{F}_t$-adapted process satisfying Lipschitz condition in $(y, z)$, i.e., there exists $\rho>0$ such that for each $y, y^{\prime} \in \mathbb{R}^d$ and $z, z^{\prime} \in \mathbb{R}^{d \times m}$,
		$$
		\left|g(t, y, z)-g\left(t, y^{\prime}, z^{\prime}\right)\right| \leqslant \rho\left(\left|y-y^{\prime}\right|+\left|z-z^{\prime}\right|\right) .
		$$
		
		\textbf{(b)} $g(\cdot, 0,0) \in \mathcal{M}_{\mathcal{F}}^2\left(0, T ; \mathbb{R}^d\right)$.
		
		\begin{lemma}\label{classic BSDE uniqueness and exsitence}
			Let $g$ satisfy (a) and (b). Then for any given terminal condition $\xi \in L^2\left(\mathcal{F}_T ; \mathbb{R}^d\right)$, the BSDE
			\begin{equation}\label{classical BSDE}
				Y_t=\xi+\int_t^T g\left(s, Y_s, Z_s\right) d s-\int_t^T Z_s d W_s, \quad 0 \leqslant t \leqslant T
			\end{equation}
			has a unique solution, i.e., there exists a unique pair of $\mathcal{F}_{\text {t}}$-adapted processes $\left(Y., Z.\right)$ $\in S_{\mathcal{F}}^2\left(0, T; \mathbb{R}^d\right) \times L_{\mathcal{F}}^2\left(0, T ; \mathbb{R}^{d \times m}\right)$ satisfying equation (\ref{classical BSDE}).
		\end{lemma}
		
		\begin{lemma}\label{classic strict BSDE comparison}
			Let $g_j(\omega, t, y, z): \Omega \times[0, T] \times \mathbb{R} \times \mathbb{R}^m \rightarrow \mathbb{R}$ satisfy $(a)$ and $(b), j=1,2$. Let $\left(Y_{.}^{(1)}, Z_{.}^{(1)}\right)$ and $\left(Y_{.}^{(2)}, Z_{.} ^{(2)}\right)$ be, respectively, the solutions of BSDEs as follows:
			$$
			Y_t^{(j)}=\xi^{(j)}+\int_t^T g_j\left(s, Y_s^{(j)}, Z_s^{(j)}\right) d s - \int_t^T Z_s^{(j)} d W_s, \quad 0 \leqslant t \leqslant T
			$$
			where $j=1,2$. If $\xi^{(1)} \geqslant \xi^{(2)}$ and $g_1\left(t, Y_t^{(1)}, Z_t^{(1)}\right) \geqslant g_2\left(t, Y_t^{(1)}, Z_t^{(1)}\right)$, a.e., a.s., then
			$$
			Y_t^{(1)} \geqslant Y_t^{(2)}, \quad \text { a.e., a.s. }
			$$
			We also have the strict comparison result: under the above conditions,
			$$
			\begin{aligned}
				& Y_0^{(1)}=Y_0^{(2)} \Longleftrightarrow \xi^{(1)}=\xi^{(2)}, \quad \text { a.s., } \\
				& g_1\left(t, Y_t^{(1)}, Z_t^{(1)}\right)=g_2\left(t, Y_t^{(1)}, Z_t^{(1)}\right), \quad \text { a.e., a.s. }
			\end{aligned}
			$$
		\end{lemma}
		
		\begin{lemma}\label{classic BSDE comparison}
			Let $g_j(\omega, t, y, z): \Omega \times[0, T] \times \mathbb{R} \times \mathbb{R}^m \rightarrow \mathbb{R}$ satisfy $(a)$ and $(b), j=1,2$. If $\xi^{(1)} \geqslant \xi^{(2)}$ and $g_1(t, y, z) \geqslant g_2(t, y, z), t \in[0, T], y \in \mathbb{R}, z \in \mathbb{R}^m$, then
			$$
			Y_t^{(1)} \geqslant Y_t^{(2)}, \quad \text { a.e., a.s. }
			$$
			
		\end{lemma}

		\section{The Formulation and Priori Estimate}\label{chapter formulation}
		\subsection{The Formulation}
		We first introduce IABSDEs  as follows:
		\begin{equation} \label{ABSDE}
			\left\{\begin{aligned}
				Y_t=&\xi_T + \int_t^{T} f\left(s, \{Y_r\}_{r \in [s,+\infty)}, \{Z_r\}_{r \in [s,+\infty)} \right) d s- \int_t^{T} Z_s d W_s, \quad &&t \in[0, T]; \\
				Y_t = & \xi_t, \quad && t\in[T, +\infty); \\
				Z_t = & \eta_t, \quad && t\in[T, +\infty).
			\end{aligned}\right.
		\end{equation}
		Here, the generator $f$ at time $s$ depends on the future values of the solutions $\{Y_r\}_{r \in [s,+\infty)}$ and $\{Z_r\}_{r \in [s,+\infty)}$. Our aim is to find a pair of $\mathcal{F}_t$-adapted processes $(Y., Z.) \in \mathcal{S}_\mathcal{F}^2\left(0, +\infty ; \mathbb{R}^d\right) \times \mathcal{M}_\mathcal{F}^{2,\beta}\left(0, +\infty; \mathbb{R}^{d \times m}\right)$ satisfying the IABSDEs (\ref{ABSDE}) $P-a.s.$ 	
		%		\begin{definition}
			%			$Y_t:\Omega \times [0, \infty)\rightarrow \mathbb{R}^d$ and $Z_t:\Omega \times [0, +\infty) \rightarrow \mathbb{R}^{d\times m}$
			%			are called a pair of solutions of IBSFDE (\ref{ABSDE}), if $\left(Y_t,Z_t\right)$ are $\mathcal{F}_t$-adapted and satisfy (\ref{ABSDE}) $P-a.s.$
			%		\end{definition}
		
		To study the existence and uniqueness of the solutions of IABSDEs (\ref{ABSDE}) and further properties,
		we assume that for all $t \in[0, T], f(t, \omega, y_{.}, z_{.}): \Omega \times \mathcal{M}_\mathcal{F}^2\left(t, +\infty ; \mathbb{R}^d\right) \times \mathcal{M}_\mathcal{F}^{2,\beta}\left(t, +\infty; \mathbb{R}^{d \times m}\right) \rightarrow L^2\left(\mathcal{F}_t; \mathbb{R}^d\right)$, and $f$ satisfies the following conditions:
		
		\textbf{(H1)} (uniform Lipschitz condition) There exists a constant $L>0$, such that for all $t \in[0, T], y_{.}, y_{.}^{\prime} \in \mathcal{M}_\mathcal{F}^2\left(t, +\infty; \mathbb{R}^d\right), z_{.}, z_{.}^{\prime} \in \mathcal{M}_\mathcal{F}^{2,\beta}\left(t, +\infty; \mathbb{R}^{d \times m}\right)$, it follows that
		$$
		E\left[\int_{t}^{T}\left|f(s,y_{.},z_{.})-f(s,y_{.}^{\prime},z_{.}^{\prime}) \right|^2 e^{\beta s} ds\right]
		\leqslant LE\left[\int_{t}^{+\infty}\left(\left\|y_s - y_s^{\prime}\right\|^2 + \left|z_s - z_s^{\prime}\right|^2\right)e^{\beta s} ds\right],
		$$
		where $\beta \geqslant 0$ is an arbitrary constant and $\|\cdot\|$ denotes the supreme norm, i.e. $\left\|a_s\right\| = \sup_{s \leqslant r \textless +\infty}\left|a_r\right|$.
		
		\textbf{(H2)}
		%\textcolor{red}{¨¨?1?¨¬??t?¨®???¨¢¨°???¨®D??$|f(s, 0,0)|\leqslant M$?¨ª?¨¦¨°????a$Y$¦Ì?D??¨º?¨®???¨¢¨®D??}
		$$
		E\left[\int_0^T|f(s, 0,0)|^2 d s\right] <+\infty.
		$$

		\subsection{Comparative Analysis of Condition (H1)}\label{chapter condition compare}
		This section shows that our condition (H1) is weaker than the corresponding conditions in \cite{peng09AP} and \cite{yangzhe13ECP}. Consider the conditions for BSDEs with finite anticipation as follows:
		\begin{itemize}
			\item \textbf{(Ha)} Condition (H1) in Peng and Yang \cite{peng09AP}: There exists a constant $L>0$, such that for all $t \in[0, T], y, y^{\prime} \in \mathbb{R}^d$, $z, z^{\prime} \in \mathbb{R}^{d \times m}, \xi_{.}, \xi_{.}^{\prime} \in \mathcal{M}_{\mathcal{F}}^2\left(t, T+K ; \mathbb{R}^d\right), \eta_{.}, \eta_{.}^{\prime} \in \mathcal{M}_{\mathcal{F}}^2\left(t, T+K ; \mathbb{R}^{d \times m}\right)$, $r, \bar{r} \in [t, T+K]$, we have
			$$
			\begin{aligned}
				& \left|f\left(t, y, z, \xi_r, \eta_{\bar{r}}\right)-f\left(t, y^{\prime}, z^{\prime}, \xi_r^{\prime}, \eta_{\bar{r}}^{\prime}\right)\right| \\
				& \quad \leqslant L\left(\left|y-y^{\prime}\right|+\left|z-z^{\prime}\right|+E^{\mathcal{F}_t}\left[\left|\xi_r-\xi_r^{\prime}\right|+\left|\eta_{\bar{r}}-\eta_{\bar{r}}^{\prime}\right|\right]\right).
			\end{aligned}
			$$
			\item \textbf{(Hb)} Condition (H4) in Yang and Elliot \cite{yangzhe13ECP}: There exists a constant $L>0$, such that for all $t \in[0, T], y_{.}, y_{.}^{\prime} \in \mathcal{M}_\mathcal{F}^2\left(t, T+K; \mathbb{R}^d\right), z_{.}, z_{.}^{\prime} \in \mathcal{M}_\mathcal{F}^{2}\left(t, T+K; \mathbb{R}^{d \times m}\right)$, it follows that for any constant $\beta \geqslant 0$
			$$
			E\left[\int_{t}^{T}\left|f(s,y_{.},z_{.})-f(s,y_{.}^{\prime},z_{.}^{\prime}) \right|^2 e^{\beta s} ds\right]
			\leqslant LE\left[\int_{t}^{T+K}\left(|y_s-y^{\prime}_s|^2 + \left|z_s - z_s^{\prime}\right|^2\right)e^{\beta s} ds\right].
			$$
			\item \textbf{(Hc)} Our condition (H1) in finite anticipation settings: There exists a constant $L>0$, such that for all $t \in[0, T], y_{.}, y_{.}^{\prime} \in \mathcal{M}_\mathcal{F}^2\left(t, T+K; \mathbb{R}^d\right), z_{.}, z_{.}^{\prime} \in \mathcal{M}_\mathcal{F}^{2}\left(t, T+K; \mathbb{R}^{d \times m}\right)$, it follows that for any constant $\beta \geqslant 0$,
			$$
			E\left[\int_{t}^{T}\left|f(s,y_{.},z_{.})-f(s,y_{.}^{\prime},z_{.}^{\prime}) \right|^2 e^{\beta s} ds\right]
			\leqslant LE\left[\int_{t}^{T+K}\left(\|y_s-y^{\prime}_s\|^2 + \left|z_s - z_s^{\prime}\right|^2\right)e^{\beta s} ds\right].
			$$
		\end{itemize}
		
		\begin{remark}
			(Ha) is stronger than (Hb) (see Remarks 3.1 in Yang and Elliot \cite{yangzhe13ECP}), and we can easily show that (Hb) is stronger than (Hc) since $\left\|y_s-y^{\prime}_s\right\| = \sup_{s \leqslant r \leqslant T+K}\left|y_r-y^{\prime}_r\right|$. In the following, we give some examples to verify conditions (Ha), (Hb) and (Hc).
		\end{remark}
		\begin{example}
			For all $t \in[0, T]$, set $f(t, y_{.}) = E^{\mathcal{F}_t}\left[\int_{t}^{t+K}y_rdWr\right]$, then for any $y_{.}, y_{.}^{\prime} \in \mathcal{M}_\mathcal{F}^2\left(t, T+K; \mathbb{R}^d\right)$, we have
			$$
			\begin{aligned}
				E\left[\int_t^T|f(s, y_{.})-f(s, y_{.}^{\prime})|^2 d s\right]
				&\leqslant \int_t^TE\left[\left|\int_{s}^{s+K}\left(y_r - y^{\prime}_r\right) dWr\right|^2\right] ds\\
				&=\int_t^TE\left[\int_{s}^{s+K}\left|y_r - y^{\prime}_r\right|^2 dr\right] ds\\
				&\leqslant KE\left[\int_{t}^{T+K}\left\|y_{s}-y^{\prime}_{s}\right\|^2 ds\right].
			\end{aligned}
			$$
			Thus, (Hc) holds.
			
			A counterexample is that, for all $t\in[0,T]$, set $ y_{t}\equiv a$ and $y^{\prime}_{t} \equiv b$ where $a,b$ are given constants, then
			$$
			f\left(t, y_{.}\right)-f(t, y_{.}^{\prime})=(a-b)E^{\mathcal{F}_t}\left[W_{t+K}-W_t\right]\sim (a-b)\sqrt{K}\cdot\mathcal{N}(0, 1),
			$$
			where $\mathcal{N}(0, 1)$ is a standard normal distribution which is unbounded. Thus $f$ satisfies conditions (Hb) and (Hc) but not (Ha).
		\end{example}
		
		\begin{remark}
			Note that (Hb) is stronger than (Hc). A natural example is that, (Hc) allows $f$ to depend on $Y$-anticipation term through the sup-norm, e.g.
			$$
			f(t, y_{.})=E^{\mathcal{F}_t}\left[\phi(\sup _{s \in[t, T+K]}\left|y_s\right|)\right],
			$$
			where $\phi(\cdot)$ is a given Lipschitz function. Moreover, the following example shows that $f$ satisfies (Hc) but not (Hb).
		\end{remark}
		
		\begin{example}
			For all $t\in[0,T]$, set $f(t, y .)=L E^{\mathcal{F}_t}\left[  y_\tau 1_{[0, \tau]}(t)\right]$,
			where $L > 0$ and $\tau \in (0,T)$ are given constants. Then for any $y_{.}, y_{.}^{\prime} \in \mathcal{M}_\mathcal{F}^2\left(t, T+K; \mathbb{R}^d\right)$,
			$$
			|f(t, y_{.})-f(t, y^{\prime}_{.})|
			= LE^{\mathcal{F}_t}\left[ \left|y_\tau-y^{\prime}_\tau\right| \cdot 1_{[0, \tau]}(t)\right]
			\leqslant L E^{\mathcal{F}_t}\left[\sup _{t \leqslant s \leqslant T+K}\left|y_s-y^{\prime}_s\right|\right].
			$$
			Thus, (Hc) holds.
			
			{A counterexample is that, for all $t\in[0,\tau)$, set $ y_{s}\equiv0$ and perturbed $y^{\prime}_{s}=\delta\cdot 1_{\{\tau\}}(s)$ where $\delta \textgreater 0$ is a given constant. Then for (Hb),
				$$
				\begin{aligned}
					E\left[ \int_t^T |f(s,y_\cdot) - f(s,y'_\cdot)|^2 ds \right]&=E\left[\int_t^T\left|0-L y^{\prime}_\tau \cdot 1_{[0, \tau]}(s)\right|^2 d s\right]
					=L^2 \delta^2 (\tau-t)>0,\\
					E\left[ \int_0^T |y_s - y'_s|^2 ds \right]&=E\left[\int_0^T\left|0-\delta \cdot 1_{\{\tau\}}(s)\right|^2 d s\right]=0.
				\end{aligned}
				$$
				Thus $f$ does not satisfy (Hb)}.
		\end{example}
		
		\subsection{Priori Estimate}
		\indent We first give the priori estimate of the solutions of IABSDEs that is used in the following sections.
		\begin{lemma}\label{lemma basic estimate}
			Let assumptions (H1) and (H2) hold, if $(Y ., Z.)$ is the solution of the IABSDE (\ref{ABSDE}), then there exists a positive constant $C_0$ that  depends  on $L$ in (H1) and $T$, such that for any given terminal condition $\xi. \in \mathcal{S}^2_{\mathcal{F}}\left(T, +\infty ; \mathbb{R}^d\right), \eta. \in \mathcal{M}_{\mathcal{F}}^{2,\beta}\left(T, +\infty; \mathbb{R}^{d \times m}\right)$, the solution $(Y ., Z.)$ satisfies
			\begin{equation}\label{basic estimate}
				\begin{aligned}
					E&\left[\sup _{0 \leqslant s \textless +\infty}\left|Y_s\right|^2+\int_0^{+\infty}e^{\beta s}\left|Z_s\right|^2 d s\right]\\
					&\leqslant C_0 E\left[\|\xi_T\|^2+\int_{T}^{+\infty}e^{\beta s}\left|\eta_s\right|^2 ds+\int_0^T|f(s,0,0)|^2 d s\right].
				\end{aligned}
			\end{equation}	
			Moreover, $(Y., Z.) \in \mathcal{S}_{\mathcal{F}}^2\left(0, +\infty; \mathbb{R}^d\right) \times \mathcal{M}_{\mathcal{F}}^{2,\beta}\left(0, +\infty; \mathbb{R}^{d \times m}\right)$.
			
		\end{lemma}	
		%\noindent\textbf{Proof.}
		\begin{proof} For $s \in[0, T]$, applying It\^{o}'s formula to $e^{\beta s}\left|Y_s\right|^2$, we obtain
			\begin{equation}\label{first ito}
				\begin{aligned}
					e^{\beta s}\left|Y_s\right|^2 & +\int_s^T e^{\beta r}\left(\beta\left|Y_r\right|^2+\left|Z_r\right|^2\right) d r \\
					= & e^{\beta T}\left|Y_T\right|^2-2 \int_s^T e^{\beta r}\left(Y_r, Z_r d W_r\right) \\
					& +2 \int_s^T e^{\beta r}\left(f\left(r, \{Y_k\}_{k\in [r,+\infty)}, \{Z_k\}_{k\in [r,+\infty)}\right), Y_r\right) d r.
				\end{aligned}
			\end{equation}
			Taking expectation on both sides of (\ref{first ito}), and note that
			\begin{equation}\label{ys}
				\begin{aligned}
					&2E\left[\int_{s}^{T}e^{\beta r}\left(f\left(r, \{Y_k\}_{k\in [r,+\infty)}, \{Z_k\}_{k\in [r,+\infty)}\right), Y_r\right)dr\right] \\
					= & 2E\left[\int_{s}^{T}e^{\beta r}\left(f\left(r, \{Y_k\}_{k\in [r,+\infty)}, \{Z_k\}_{k\in [r,+\infty)}\right)-f\left(r, 0, 0\right)+f\left(r, 0, 0\right), Y_r\right)dr \right]\\
					\leqslant & E\left[\int_{s}^{T}e^{\beta r}\left(2L\left|Y_r\right|^2 + \frac{1}{2L}\left|f\left(r, \{Y_k\}_{k\in [r,+\infty)}, \{Z_k\}_{k\in [r,+\infty)}\right)-f\left(r, 0, 0\right)\right|^2\right) dr \right]\\
					&+2E\left[\int_{s}^{T}e^{\beta r}\left(f\left(r, 0, 0\right), Y_r\right)dr \right]\\
					\leqslant & E\left[2L\int_{s}^{T}e^{\beta r}\left|Y_r\right|^2 dr + \frac{1}{2} \int_s^T e^{\beta r}E^{\mathcal{F}_r}\left[\|Y_r\|^2\right] dr
					+ \frac{1}{2}\int_{s}^{+\infty}e^{\beta r}\left|Z_r\right|^2 dr \right]\\
					&+ 2E\left[\int_{s}^{T}e^{\beta r}\left(f\left(r, 0, 0\right), Y_r\right)dr\right].
				\end{aligned}
			\end{equation}
			We obtain for $s \in [0,T]$,
			$$
			\begin{aligned}
				E\left[e^{\beta s}\left|Y_s\right|^2\right]& +E\left[\int_s^T e^{\beta r}\left[\left(\beta-2L\right)\left|Y_r\right|^2+\frac{1}{2}\left|Z_r\right|^2\right] d r \right]\\
				\leqslant & E\left[e^{\beta T}\|\xi_T\|^2 +\frac{1}{2}\int_{T}^{+\infty}e^{\beta r}\left|\eta_r\right|^2 dr + 2\int_s^T e^{\beta r} \left(f(r, 0,0), Y_r\right)dr\right]\\
				&+ \frac{1}{2} \int_s^T e^{\beta r}E\left[\|Y_r\|^2\right] dr.
			\end{aligned}
			$$
			Set $\beta = 2L + 2$, we have the estimate over the $Z$ component, %by setting $s=0$,
			\begin{equation}\label{13}
				\begin{aligned}
					E&\left[\int_s^T e^{\beta r}\left|Z_r\right|^2 d r\right]\\
					\leqslant & E\left[2e^{\beta T}\|\xi_T\|^2 + \int_{T}^{+\infty}e^{\beta r}\left|\eta_r\right|^2 dr + 4\int_s^T e^{\beta r} \left(f(r, 0,0), Y_r\right)dr\right] + \int_s^T e^{\beta r}E\left[\|Y_r\|^2\right] dr.
				\end{aligned}
			\end{equation}
			For the martingale part in (\ref{first ito}), notice that for $s \in [t,T]$,
			$$
			\begin{aligned}
				\left|\int_s^T e^{\beta r}\left(Y_r, Z_r d W_r\right)\right| & =\left|\int_t^T e^{\beta r}\left(Y_r, Z_r d W_r\right)-\int_t^s e^{\beta r}\left(Y_r, Z_r d W_r\right)\right| \\
				& \leqslant\left|\int_t^T e^{\beta r}\left(Y_r, Z_r d W_r\right)\right|+\left|\int_t^s e^{\beta r}\left(Y_r, Z_r d W_r\right)\right|,
			\end{aligned}
			$$
			by the Burkholder-Davis-Gundy inequality, we have
			\begin{equation}\label{14}
				\begin{aligned}
					E & {\left[\sup _{t \leqslant s \leqslant T}\left|\int_s^T e^{\beta r}\left(Y_r, Z_r d W_r\right)\right|\right] } \\
					& \leqslant 2 E\left[\sup _{t \leqslant s \leqslant T}\left|\int_t^s e^{\beta r}\left(Y_r, Z_r d W_r\right)\right|\right] \\
					& \leqslant 6 E\left[\left(\int_t^T e^{2 \beta r}\left|Y_r\right|^2\left|Z_r\right|^2 d r\right)^{1 / 2}\right] \\
					& \leqslant 6 E\left[\left(\sup _{t \leqslant r \leqslant T} e^{\beta r / 2}\left|Y_r\right|\right)\left(\int_t^T e^{\beta r}\left|Z_r\right|^2 d r\right)^{1 / 2}\right] \\
					& \leqslant \frac{1}{4} E\left[\sup _{t \leqslant r \leqslant T} e^{\beta r}\left|Y_r\right|^2\right]+36 E\left[\int_t^T e^{\beta r}\left|Z_r\right|^2 d r\right].
				\end{aligned}
			\end{equation}
			Combining estimates (\ref{first ito}), (\ref{13}) and (\ref{14}), it follows that
			\begin{equation} \label{middle}
				\begin{aligned}
					(\beta- 2 &L)E \left[\sup _{t \leqslant s \leqslant T} e^{\beta s}\left|Y_s\right|^2\right] \\
					\leqslant & E\left[e^{\beta T}\|\xi_T\|^2 + \frac{1}{2}\int_{T}^{+\infty}e^{\beta r}\left|\eta_r\right|^2 dr + 2 \int_t^T e^{\beta r}\left|f(r,0,0) \| Y_r\right| d r\right] \\
					& +2E\left[ \sup _{t \leqslant s \leqslant T}\left|\int_s^T e^{\beta r}\left(Y_r, Z_r d W_r\right)\right|\right] + \frac{1}{2}\int_t^T e^{\beta r} E\left[\|Y_r\|^2\right] d r \\
					\leqslant & E\left[e^{\beta T}\|\xi_T\|^2 +\frac{1}{2}\int_{T}^{+\infty}e^{\beta r}\left|\eta_r\right|^2 dr + 2 \int_t^T e^{\beta r}\left|f(r,0,0) \| Y_r\right| d r\right]\\
					& +\frac{1}{2} E\left[\sup _{t \leqslant r \leqslant T} e^{\beta r}\left|Y_r\right|^2\right] +72 E\left[\int_t^T e^{\beta r}\left|Z_r\right|^2 d r\right]
					+ \frac{1}{2}\int_t^T e^{\beta r} E\left[\|Y_r\|^2\right] d r\\
					\leqslant & E\left[(1+72*2)e^{\beta T}\|\xi_T\|^2 + (\frac{1}{2}+72)\int_{T}^{+\infty}e^{\beta r}\left|\eta_r\right|^2 dr\right]\\
					& + (2+72*4)E\left[ \int_t^T e^{\beta r}\left|f(r,0,0) \| Y_r\right| d r\right]\\
					& +\frac{1}{2} E\left[\sup _{t \leqslant r \leqslant T} e^{\beta r}\left|Y_r\right|^2\right] + (\frac{1}{2}+72)\int_t^T e^{\beta r} E\left[\|Y_r\|^2\right] d r,
				\end{aligned}
			\end{equation}
			where
			\begin{equation}\label{yf}
				\begin{aligned}
					E&\left[\int_t^T e^{\beta r}\left|f(r,0,0) \| Y_r\right| d r\right] \\
					& \leqslant E\left[\int_t^T \left(\sup_{t\leqslant r \leqslant T}e^{\beta r/2}\left| Y_r\right|\right)\left(e^{\beta r/2}\left|f(r,0,0)\right|\right)  d r\right] \\
					& \leqslant \frac{1}{4+72*8}E\left[\sup _{t \leqslant r \leqslant T} e^{\beta r}\left|Y_r\right|^2\right] + (2+72*4)Te^{\beta T}E\left[\int_t^T|f(s,0,0)|^2 d s\right].
				\end{aligned}
			\end{equation}
			Since $\|Y_r\|^2 \leqslant \left(\sup_{r\leqslant k\leqslant T}\left|Y_k\right| + \sup_{k\geqslant T}\left|Y_k\right|\right)^2$, then \\
			\begin{equation}\label{bc}
				\begin{aligned}
					\int_t^T& e^{\beta r} E\left[\|Y_r\|^2\right] d r \\
					& \leqslant 2 \int_t^T e^{\beta r}E\left[ \sup_{r\leqslant k\leqslant T} \left|Y_k\right|^2 + \|\xi_{T}\|^2 \right] d r\\
					& \leqslant 2 \int_t^T E\left[ \sup_{r\leqslant k\leqslant T}e^{\beta k} \left|Y_k\right|^2  \right] d r + 2Te^{\beta T}E\left[\|\xi_{T}\|^2\right].
				\end{aligned}
			\end{equation}
			Donote by $C_{0}$ a positive constant that depend only on $L$ and $T$, which we allow to change from line to line. Combining estimates (\ref{middle}), (\ref{yf}) and (\ref{bc}), we have
			$$
			\begin{aligned}
				(\beta - 2L - 1)E& \left[\sup _{t \leqslant s \leqslant T} e^{\beta s}\left|Y_s\right|^2\right] \\
				\leqslant & C_0 E\left[\|\xi_T\|^2+\int_{T}^{+\infty}e^{\beta r}\left|\eta_r\right|^2 dr+\int_t^T|f(s,0,0)|^2 d s \right]\\
				& +C_0 \int_t^T E\left[\sup _{r \leqslant k \leqslant T} e^{\beta k}\left|Y_k\right|^2\right] d r.
			\end{aligned}
			$$
			Since $\beta=2L+2$, then by the Gronwall's inequality, we obtain the estimate over $Y$ component
			\begin{equation}\label{finite estimate}
				\begin{aligned}
					E&\left[\sup _{t \leqslant s \leqslant T}e^{\beta s}\left|Y_s\right|^2\right] \leqslant C_0 E\left[\|\xi_T\|^2+\int_{T}^{+\infty}e^{\beta r}\left|\eta_r\right|^2 dr+ \int_t^T|f(s,0,0)|^2 d s\right].
				\end{aligned}
			\end{equation}
			Setting $t=0$ and noticing $\sup_{t \leqslant s \textless +\infty}\left|Y_s\right|^2 \leqslant \sup_{t \leqslant s \leqslant T}\left|Y_s\right|^2 + \|\xi_T\|^2 $, we obtain
			$$
			\begin{aligned}
				E&\left[\sup _{0 \leqslant s \textless +\infty}\left|Y_s\right|^2+\int_0^{+\infty}e^{\beta s}\left|Z_s\right|^2 d s\right]\\
				&\leqslant E\left[\sup _{0 \leqslant s \leqslant T}\left|Y_s\right|^2 + \|\xi_T\|^2 + \int_0^{T}e^{\beta s}\left|Z_s\right|^2 d s +\int_{T}^{+\infty}e^{\beta r}\left|\eta_r\right|^2 dr\right]\\
				&\leqslant C_0 E\left[\|\xi_T\|^2+\int_{T}^{+\infty}e^{\beta r}\left|\eta_r\right|^2 dr+\int_0^T|f(s,0,0)|^2 d s\right] \textless +\infty.
			\end{aligned}
			$$
			It is obvious that $Y$ has continuous path, thus  $(Y., Z.) \in \mathcal{S}_{\mathcal{F}}^2\left(0, +\infty; \mathbb{R}^d\right) \times \mathcal{M}_{\mathcal{F}}^{2,\beta}\left(0, +\infty; \mathbb{R}^{d \times m}\right)$. %And estimate (\ref{finite estimate})  guarantees that $\sup_{0 \leqslant t \leqslant T}\left|Y_t\right|\textless +\infty$ a.s., thus $(Y., Z.) \in \mathcal{S}_{\mathcal{F}}^2\left(0, +\infty; \mathbb{R}^d\right) \times \mathcal{M}_{\mathcal{F}}^{2,\beta}\left(0, +\infty; \mathbb{R}^{d \times m}\right)$.
		\end{proof}
		%\hfill $\qedsymbol$
		\\
		\section{The Existence and Uniqueness Theorem}\label{chapter uniqueness and existence}
		In this section, we investigate the main result of this paper: the existence and uniqueness theorem for adapted solutions of IABSDEs.
		
		\begin{theorem}\label{ex and uni}
			Suppose that $f$ satisfies assumptions (H1) and (H2), then for any given terminal condition $\xi.\in\mathcal{S}_\mathcal{F}^2\left(T, +\infty; \mathbb{R}^d\right)$, $\eta.\in\mathcal{M}_{\mathcal{F}}^{2,\beta}\left(T, +\infty; \mathbb{R}^{d \times m}\right)$, the IABSDE (\ref{ABSDE}) has a unique solution, that is, there exists a unique pair of $\mathcal{F}_t$-adapted processes $(Y., Z.) \in \mathcal{S}_\mathcal{F}^2\left(0, +\infty; \mathbb{R}^d\right) \times \mathcal{M}_\mathcal{F}^{2,\beta}\left(0, +\infty; \mathbb{R}^{d \times m}\right)$ satisfying (\ref{ABSDE}).
		\end{theorem}
		%\noindent\textbf{Proof.~~}\\
		\begin{proof} \textbf{First, we verify the uniqueness of the solutions. }
			
			Let $\left(Y.,Z.\right)$ and $\left(\bar{Y}.,\bar{Z}.\right)$ be any two solutions of the IABSDE (\ref{ABSDE}). By Lemma \ref{lemma basic estimate},  $(Y_{.}, Z_{.}), (Y_{.}^{\prime}, Z_{.}^{\prime}) \in \mathcal{S}_{\mathcal{F}}^2\left(0, +\infty; \mathbb{R}^d\right) \times \mathcal{M}_{\mathcal{F}}^{2,\beta}\left(0, +\infty; \mathbb{R}^{d \times m}\right)$. Denote their differences by $\left(\tilde{Y}_{.}, \tilde{Z}_{.}\right)=\left(Y_{.}-Y_{.}^{\prime},Z_{.}-Z_{.}^{\prime}\right)$. Note that $\left(\tilde{Y}_{.}, \tilde{Z}_{.}\right)$ satisfies	
			$$
			\tilde{Y}_t=\int_t^{T} \tilde{f}_s d s- \int_t^{T} \tilde{Z}_s dW_s, \quad t \in[0, T].
			$$
			where
			$$
			\tilde{f}_t =f\left(t,\left\{Y_r\right\}_{r \in [t,+\infty)},\left\{Z_r\right\}_{r \in [t,+\infty)}\right)-f\left(t,\left\{\bar{Y}_r\right\}_{r \in [t,+\infty)},\left\{\bar{Z}_r\right\}_{r \in [t,+\infty)}\right).
			$$
			Applying It\^{o}'s formula to $e^{\beta s}\left|\tilde{Y}_s\right|^2$ for $s \in[0, T]$, we obtain
			\begin{equation}\label{kkkk}
				\begin{aligned}
					e^{\beta s}\left|\tilde{Y}_s\right|^2 & +\int_s^T e^{\beta r}\left(\beta\left|\tilde{Y}_r\right|^2+\left|\tilde{Z}_r\right|^2\right) d r \\
					= & 2 \int_s^T e^{\beta r}\left(\tilde{f}_r, \tilde{Y}_r\right) d r-2 \int_s^T e^{\beta r}\left(\tilde{Y}_r, \tilde{Z}_r d W_r\right) .
				\end{aligned}
			\end{equation}
			Similar with estimates (\ref{ys}), (\ref{13}) and (\ref{14}), we have for all $t\in[0,T]$
			$$
			\begin{aligned}
				2 E\left[\int_s^T e^{\beta r}\left(\tilde{f}_r, \tilde{Y}_r\right) d r\right]
				\leqslant  E\left[\int_s^T e^{\beta r}\left(2L\left|\tilde{Y}_r\right|^2 + \frac{1}{2}\left|\tilde{Z}_r\right|^2 + \frac{1}{2}E^{\mathcal{F}_r}\left[\|\tilde{Y}_r\|^2\right]\right)dr\right],
			\end{aligned}
			$$
			and
			$$
			\begin{aligned}
				E {\left[\sup _{t \leqslant s \leqslant T}\left|\int_s^T e^{\beta r}\left(\tilde{Y}_r, \tilde{Z}_r d W_r\right)\right|\right] }
				\leqslant& \frac{1}{4} E\left[\sup _{t \leqslant r \leqslant T} e^{\beta r}\left|\tilde{Y}_r\right|^2\right]+36 E\left[\int_t^T e^{\beta r}\left|\tilde{Z}_r\right|^2 d r\right],
			\end{aligned}
			$$
			and the estimate for $\tilde{Z}$
			$$
			\begin{aligned}
				E\left[\int_t^T e^{\beta r}\left|\tilde{Z}_r\right|^2 d r\right]
				\leqslant \int_t^T e^{\beta r}E\left[\|\tilde{Y}_r\|^2\right] dr.
			\end{aligned}
			$$
			Combining (\ref{kkkk}) and the estimates mentioned above, we have
			\begin{equation}\label{1234}
				\begin{aligned}
					(\beta - 2L - 73)E \left[\sup _{t \leqslant s \leqslant T} e^{\beta s}\left|\tilde{Y}_s\right|^2\right]
					\leqslant& (\frac{1}{2}+72)\int_0^T e^{\beta r}E\left[\|\tilde{Y}_r\|^2\right] dr \\
					\leqslant& (\frac{1}{2}+72)\int_t^T  E\left[\sup_{r \leqslant k \leqslant T}e^{\beta k}\left|\tilde{Y}_k\right|^2\right] d r.
				\end{aligned}
			\end{equation}
			{Set $\beta=2L+74$, and apply Gronwall's inequality to yield for all $t\in[0,T]$
				$$
				E\left[\sup _{t \leqslant s \leqslant T}|\tilde{Y}_s|^2\right]=0.
				$$
				Set $t=0$, it is obvious that
				$$
				E\left[\sup _{0 \leqslant s \leqslant T}|\tilde{Y}_s|^2\right] + E \left[\int_{0}^T e^{\beta s}\left|\tilde{Z}_s\right|^2 d s\right] = 0.
				$$}
			That is $\left(Y_{.},Z_{.}\right) =\left(Y_{.}^{\prime},Z_{.}^{\prime}\right)$.
			\\
			\\
			\textbf{Next, we show the existence of the solutions.}
			
			\noindent Define $Y_{t}^{(0)}=\xi_{t}$, $Z_{t}^{(0)}=\eta_{t}$ for $t \in [T,+\infty)$ and $Y_t^{(0)}=Z_t^{(0)}=0$, for $t\in [0,T]$. For $n=1,2, \ldots$, define the Picard sequence through
			\begin{equation}\label{picard}
				\left\{ \begin{aligned}
					Y_t^{(n)}=&\xi_T+\int_{t}^T f\left(s, \{Y_r^{(n-1)}\}_{r \in [s,+\infty)}, \{Z_r^{(n)}\}_{r \in [s,+\infty)}\right) d s-\int_{t}^T Z_s^{(n)} d W_s, \quad &&t\in [0,T];\\
					Y_t^{(n)}=&\xi_{t}, \quad &&t\in [T, +\infty);\\
					Z_t^{(n)}=&\eta_{t}, \quad &&t\in [T, +\infty).
				\end{aligned}\right.
			\end{equation}
			Obviously, $\left(Y_{.}^{(0)}, Z_{.}^{(0)}\right) \in \mathcal{S}_{\mathcal{F}}^2\left(0,+\infty ; \mathbb{R}^d\right)\times \mathcal{M}_{\mathcal{F}}^{2,\beta}\left(0,+\infty ; \mathbb{R}^{d \times m}\right)$. Similarly to (\ref{basic estimate}), it is easy to prove that
			$$
			E\left[\sup _{0 \leqslant t \textless +\infty}|Y_t^{(n)}|^2 +
			\int_0^{+\infty} e^{\beta t}\left|Z_t^{(n)}\right|^2 d t
			\right] < +\infty,
			\quad n=1,2, \ldots
			$$
			which shows that $\left\{\left(Y_{.}^{(n)}, Z_{.}^{(n)}\right) \right\}_{n \geqslant 1}$ is a sequence in $\mathcal{S}_{\mathcal{F}}^2\left(0,+\infty ; \mathbb{R}^d\right) \times\mathcal{M}_{\mathcal{F}}^{2,\beta}\left(0,+\infty ; \mathbb{R}^{d \times m}\right)$.
			Consider the equation
			$$
			\begin{aligned}
				Y_t^{(n)}-Y_t^{(n-1)}=&\int_{t}^T f\left(s, \{Y_r^{(n-1)}\}_{r \in [s,+\infty)}, \{Z_r^{(n)}\}_{r \in [s,+\infty)}\right)d s\\
				&-\int_{t}^T f\left(s, \{Y_r^{(n-2)}\}_{r \in [s,+\infty)}, \{Z_r^{(n-1)}\}_{r \in [s,+\infty)}\right)d s-\int_{t}^T \left(Z_s^{(n)}-Z_s^{(n-1)}\right) d W_s.
			\end{aligned}
			$$
			Analogous to (\ref{1234}), set $\beta=2L+74$, there exists a positive constant $C$ %\textcolor{red}{?y¨¨?(\ref{1234})?¨´¨º?¡ê??a¨¤?¦Ì?$C$??o??¨¦¨°?2?¨°¨¤¨¤¦Ì¨®¨²¨¨?o?2?¨ºy},
			such that
			$$
			\begin{aligned}
				E \left[\sup_{0\leqslant s\leqslant T}e^{\beta s}\left|Y_s^{(n)}-Y_s^{(n-1)}\right|^2\right] + &
				E\left[ \int_0^T e^{\beta s}\left|Z_s^{(n)}-Z_s^{(n-1)}\right|^2 ds\right]\nonumber\\
				\leqslant& C \int_0^T e^{\beta s} E\left[\|Y_s^{(n-1)}-Y_s^{(n-2)}\|^2 \right]d s \\
				\leqslant& C \int_0^T E\left[\sup_{s\leqslant r\leqslant T}e^{\beta r}\left|Y_r^{(n-1)}-Y_r^{(n-2)}\right|^2\right] d s.
			\end{aligned}
			$$
			Define $u^n(t)=E \left[\sup_{t\leqslant s\leqslant T}e^{\beta s}\left|Y_s^{(n)}-Y_s^{(n-1)}\right|^2\right]$ and
			$v^n(t)=E\left[\int_t^T e^{\beta s} \left|Z_s^{(n)}-Z_s^{(n-1)}\right|^2 ds\right]$, the  formula above can be rewritten as
			\begin{equation}\label{4444}
				u^n(0)+v^n(0)\leqslant C\int_{0}^{T}u^{n-1}(s) d s, \quad u^{n}(T)=0,\quad n=2,3, \ldots
			\end{equation}
			By iterating the above inequality, we can deduce that
			$$
			u^n(0)\leqslant \frac{(TC)^{n-1}}{(n-1)!}u^1(0),
			$$
			and
			$$
			v^n(0)\leqslant \frac{T^{n-1}C^{n-2}}{(n-2)!}u^1(0),\quad
			$$
			where we can easily show that  $u^1(0)\textless +\infty$, which implies that $\left\{\left(Y_{.}^{(n)}, Z_{.}^{(n)}\right)\right\}_{n \geqslant 1}$ is a Cauchy sequence in Banach space $\mathcal{S}_{\mathcal{F}}^2\left(0,+\infty ; \mathbb{R}^d\right) \times$ $\mathcal{M}_{\mathcal{F}}^{2,\beta}\left(0,+\infty ; \mathbb{R}^{d \times m}\right)$.  Then from (\ref{4444}),  $\left\{Y^n\right\}_{n \geqslant 1}$ is also a Cauchy sequence in $L^2\left(\Omega; C\left([0, +\infty) ; \mathbb{R}^d\right) \right)$. Passing the limit in (\ref{picard}), as $n \rightarrow \infty$, we obtain that the pair  $\left(Y., Z.\right)$ defined by
			$$
			Y_{t} = \lim _{n \rightarrow \infty} Y_{t}^{(n)}, \quad Z_{t} = \lim _{n \rightarrow \infty} Z_{t}^{(n)}
			$$
			solves equation (\ref{ABSDE}).
		\end{proof}
		%\hfill $\qedsymbol$
		\\
		\section{Comparison Theorems}\label{chapter comparison}
		
		The comparison theorem, as a fundamental tool in BSDE theory, was first obtained for classical 1-dimensional BSDEs in Peng \cite{Peng1992}.
		%We refer to \cite{BSDEinfinance, Peng2004} for further development and application in mathematical finance and risk management.
		%Developments in this direction such as El Karoui, Peng and Quenez \cite{BSDEinfinance} and Peng \cite{Peng2004} significantly expanded its applications in mathematical finance and risk management.
		For BSDEs with finite anticipation, the comparison theorem requires stronger assumptions. Imposing a monotonicity condition on $f$, Peng and Yang \cite{peng09AP} established the comparison theorems for 1-dimensional anticipated BSDEs independent of $Z$-anticipation term,
		%Then Xu \cite{Xu2016} removed the monotonicity restriction and allowed for $Z$-anticipation term,
		and see \cite{21SPL, Xu2011, Xu2016, yangzhe13SPA} for further development in this direction.
		%Subsequent breakthroughs such as Xu \cite{Xu2016} removed the monotonicity restriction and allowed for $Z$-anticipation term, and Xu \cite{Xu2011} studied the necessary and sufficient conditions for the comparison theorem in multidimensional settings.
		%Some other development, such as the converse comparison theorem, can be found in Yang and Elliot \cite{yangzhe13SPA} for 1-dimensional case, then Hao and Li \cite{21SPL} for multidimensional case.
		
		In this section, we extend the comparison theorems in Peng and Yang \cite{peng09AP} to the infinite anticipation settings without imposing additional technical conditions. We will apply these results to solve an optimal control problem in Section \ref{chapter Stochastic control problem}.
		%We will apply the comparison theorems for BSDEs with infinite anticipation to solve a optimal control problem governed by SDEs with infinite delay in Section \ref{chapter Stochastic control problem}.
		\begin{theorem}\label{comparision}
			Let $\left(Y_{.}^{(1)}, Z_{.}^{(1)}\right)$ and $\left(Y_{.}^{(2)}, Z_{.}^{(2)}\right)$ be, respectively, the solutions of the following two 1-dimensional IABSDEs:
			$$
			\left\{\begin{aligned}
				Y_t^{(j)} & =\xi_T^{(j)} + \int_t^{T} f_{j}\left(s, \{Y_r^{(j)}\}_{r \in [s,+\infty)}, Z_s^{(j)} \right) d s- \int_t^{T} Z_s^{(j)} d W_s, & & t \in[0, T]; \\
				Y_t^{(j)} & =\xi_t^{(j)}, && t\in [T, +\infty)
			\end{aligned}\right.
			$$
			where $j=1,2$. Assume that for $j=1,2$, $f_j$ satisfies (H1) and (H2), $\xi_{.}^{(j)} \in \mathcal{S}_{\mathcal{F}}^2\left(T, +\infty\right)$, and for all $t \in[0, T], z \in \mathbb{R}^m, f_2(t, \cdot, z)$ is increasing, i.e., $f_2(t, y., z) \geqslant$ $f_2\left(t, y.^{\prime}, z\right)$, if $y_r \geqslant y_r^{\prime}$, $y., y.^{\prime} \in \mathcal{S}_{\mathcal{F}}^2\left(t, +\infty \right), r \in[t, +\infty)$. If $\xi_s^{(1)} \geqslant \xi_s^{(2)}, s \in[T, +\infty)$, and $f_1(t, y., z) \geqslant f_2(t, y., z), t \in[0, T], y. \in \mathcal{S}_{\mathcal{F}}^2\left(t, +\infty\right), z \in \mathbb{R}^m$, then
			$$
			Y_t^{(1)} \geqslant Y_t^{(2)}, \quad \text { a.e., a.s. }
			$$
		\end{theorem}
		%\noindent\textbf{Proof.}
		\begin{proof}
			Set
			$$
			\left\{\begin{aligned}
				Y_t^{(3)} = & \xi_T^{(2)}+\int_t^T f_2\left(s,\{Y_r^{(1)}\}_{r \in [s,+\infty)}, Z_s^{(3)}\right) d s-\int_t^T Z_s^{(3)} d W_s, \quad &&t\in[0, T]; \\
				Y_t^{(3)} = & \xi_t^{(2)}, \quad &&t\in[T, +\infty),
			\end{aligned}\right.
			$$
			by Lemma \ref{classic BSDE uniqueness and exsitence}, there exists a unique pair of $\mathcal{F}_t$-adapted processes $\left(Y_{.}^{(3)}, Z_{.}^{(3)}\right)$ $\in \mathcal{S}_{\mathcal{F}}^2(0, +\infty) \times \mathcal{M}_{\mathcal{F}}^2\left(0, T ; \mathbb{R}^m\right)$ satisfying the above classical BSDE. Since for all $s \in[0, T], z \in \mathbb{R}^m$,
			$$
			f_1\left(s, \{Y_r^{(1)}\}_{r \in [s,+\infty)}, z\right) \geqslant f_2\left(s, \{Y_r^{(1)}\}_{r \in [s,+\infty)}, z\right),
			$$
			by Lemma \ref{classic BSDE comparison}, we get
			$$
			Y_t^{(1)} \geqslant Y_t^{(3)}, \quad \text { a.e., a.s. }
			$$
			Set
			$$
			\left\{\begin{aligned}
				Y_t^{(4)} = & \xi_T^{(2)}+\int_t^T f_2\left(s, \{Y_r^{(3)}\}_{r \in [s,+\infty)}, Z_s^{(4)}\right) d s - \int_t^T Z_s^{(4)} d W_s, && t \in[0, T]; \\ Y_t^{(4)} = & \xi_t^{(2)}, && t \in[T, +\infty).
			\end{aligned}\right.
			$$
			Since $f_2$ is increasing with respect to the $Y$-anticipation term and $Y_t^{(1)} \geqslant Y_t^{(3)}$, a.e., a.s., then also by Lemma \ref{classic BSDE comparison},
			$$
			Y_t^{(3)} \geqslant Y_t^{(4)},\quad  \text { a.e., a.s. }
			$$
			For $n=5,6, \cdots$, define the Picard sequence through the following classical BSDEs:
			$$
			\left\{\begin{aligned}
				Y_t^{(n)} = & \xi_T^{(2)} + \int_t^T f_2\left(s, \{Y_r^{(n-1)}\}_{r \in [s,+\infty)}, Z_s^{(n)}\right) d s - \int_t^T Z_s^{(n)} d W_s, && t \in[0, T]; \\
				Y_t^{(n)} = &\xi_t^{(2)}, && t \in[T, +\infty).
			\end{aligned}\right.
			$$
			Similarly, we have $Y_t^{(4)} \geqslant Y_t^{(5)} \geqslant \cdots \geqslant  Y_t^{(n)} \geqslant \cdots$ a.e., a.s.
			
			The Picard sequence defined above is essentially identical to (\ref{picard}), then without loss of generality, we use the same denotation $\left\{u^n(\cdot), v^n(\cdot)\right\}$ as in the proof the Theorem \ref{ex and uni} to show that: $\left\{u^n(t),v^n(t);0 \leqslant t \leqslant T\right\}_{n \geqslant 5}$ converges and $\left\{\left(Y_{.}^n, Z_{.}^n\right)\right\}_{n \geqslant 5}$ is a Cauchy sequence in Banach space $\mathcal{S}_\mathcal{F}^2\left(0,+\infty\right) \times$ $\mathcal{M}_\mathcal{F}^2\left(0,T ; \mathbb{R}^{m}\right)$.
			Denote the limit of $\left\{Y_{.}^n, Z_{.}^n\right\}$ by $\left\{Y_{.}, Z_{.}\right\} \in \mathcal{S}_\mathcal{F}^2\left(0,+\infty \right) \times$ $\mathcal{M}_\mathcal{F}^2\left(0,T ; \mathbb{R}^{m}\right)$. Note that for all $t\in[0,T]$,
			$$
			\begin{aligned}
				E&\left[\int_t^T\left|f_2\left(s, \{Y_r^{(n-1)}\}_{r \in [s,+\infty)}, Z_s^{(n)}\right)-f_2\left(s, \{Y_r\}_{r \in [s,+\infty)}, Z_s\right)\right|^2 e^{\beta s} d s\right] \\
				& \quad \leqslant L E\left[\int_t^T\left(E^{\mathcal{F}_s}\left[\|Y_s^{(n-1)}-Y_s
				\|\right]^2 +\left|Z_s^{(n)}-Z_s\right|^2\right) e^{\beta s} d s\right]\\
				& \quad \leqslant TL E\left[\sup_{t \leqslant s \leqslant T}\left|Y_s^{(n-1)}-Y_s\right|^2 e^{\beta s} \right]
				+ LE\left[\int_t^T \left|Z_s^{(n)}-Z_s\right|^2 e^{\beta s} d s\right] \rightarrow 0
			\end{aligned}
			$$
			when $n \rightarrow \infty$. Therefore, $\left(Y_{.}, Z_{.}\right)$ satisfies the following IABSDE:
			$$
			\left\{\begin{aligned}
				Y_t & =\xi_T^{(2)} + \int_t^{T} f_{2}\left(s, \{Y_r\}_{r \in [s,+\infty)}, Z_s \right) d s- \int_t^{T} Z_s d W_s, & & t \in[0, T]; \\
				Y_t & =\xi_t^{(2)}, && t\in [T, +\infty).
			\end{aligned}\right.
			$$
			By Theorem \ref{ex and uni}, we have
			$$
			Y_t=Y_t^{(2)}, \quad \text { a.e., a.s. }
			$$
			Therefore
			$$
			Y_t^{(1)} \geqslant Y_t^{(3)} \geqslant Y_t^{(4)} \geqslant  Y_t = Y_t^{(2)}, \quad \text { a.e., a.s. }
			$$
		\end{proof}
		%\hfill $\qedsymbol$
		
		By imposing a strict monotonicity condition on the generator, we can further obtain a strict comparison theorem.
		\begin{theorem}\label{strict comparision}
			Under the assumptions in Theorem \ref{comparision}, if $\xi_s^{(1)} \geqslant \xi_s^{(2)}, s \in[T, +\infty)$, and \\ $f_1\left(t, \{Y_r^{(1)}\}_{r \in [t,+\infty)}, Z_t^{(1)}\right) \geqslant f_2\left(t, \{Y_r^{(1)}\}_{r \in [t,+\infty)}, Z_t^{(1)}\right), t \in[0, T]$, then
			$$
			Y_t^{(1)} \geqslant Y_t^{(2)}, \quad \text { a.e., a.s. }
			$$
			We also have a strict comparison result: given the assumptions of Theorem \ref{comparision}, and suppose for all $t \in[0, T], z \in \mathbb{R}^m, f_2(t, \cdot, z)$ is strictly increasing. Then the following equivalence holds:
			$$
			Y_0^{(1)}=Y_0^{(2)} \Longleftrightarrow \left\{\begin{aligned}&f_1\left(t,  \{Y_r^{(1)}\}_{r \in [t,+\infty)}, Z_t^{(1)}\right)  =f_2\left(t, \{Y_r^{(1)}\}_{r \in [t,+\infty)}, Z_t^{(1)}\right), &&t \in[0, T];\\ &\xi_t^{(1)}=\xi_t^{(2)},   &&t \in[T, +\infty) .\end{aligned}\right.
			$$
		\end{theorem}
		%\noindent\textbf{Proof.}
		\begin{proof}
			In the first step, set
			$$
			\left\{\begin{aligned}
				Y_t^{(3)} = & \xi_T^{(2)}+\int_t^T f_2\left(s,\{Y_r^{(1)}\}_{r \in [s,+\infty)}, Z_s^{(3)}\right) d s-\int_t^T Z_s^{(3)} d W_s, \quad &&t\in[0, T]; \\
				Y_t^{(3)} = & \xi_t^{(2)}, \quad &&t\in[T, +\infty).
			\end{aligned}\right.
			$$
			Set $\tilde{f_t}=f_1\left(t, \{Y_r^{(1)}\}_{r \in [t,+\infty)}, Z_t^{(1)}\right)-f_2\left(t, \{Y_r^{(1)}\}_{r \in [t,+\infty)}, Z_t^{(1)}\right) \geqslant 0$, $\tilde{\xi}_{.}=\xi_{.}^{(1)}-\xi_{.}^{(2)} \geqslant 0$ a.e., a.s. and $\left(\tilde{Y}_{.},\tilde{Z}_{.}\right) = \left(Y_{.}^{(1)}- Y_{.}^{(3)}, Z_{.}^{(1)}-Z_{.}^{(3)}\right)$. Then the pair $\left(\tilde{Y}_{.},\tilde{Z}_{.}\right)$ is the solution of the linear BSDE:
			$$
			\left\{\begin{aligned}
				\tilde{Y}_t&=\tilde{\xi}_T+\int_t^T\left(b_s \tilde{Z}_s+\tilde{f}_s\right) d s-\int_t^T \tilde{Z}_s d W_s, \quad && t \in[0, T] ; \\
				\tilde{Y}_t&=\tilde{\xi}_t, \quad &&t \in[T, +\infty),
			\end{aligned}\right.
			$$
			where
			$$
			\begin{aligned}
				b_s= \left\{\begin{aligned}&\frac{f_2\left(s, \{Y_r^{(1)}\}_{r \in [s,+\infty)}, Z_s^{(1)}\right)-f_2\left(s, \{Y_r^{(1)}\}_{r \in [s,+\infty)}, Z_s^{(3)}\right)}{Z_s^{(1)}-Z_s^{(3)}}, \quad&& \text { if } Z_s^{(1)} \neq Z_s^{(3)}; \\
					&0,  \quad&&\text { if } Z_s^{(1)}=Z_s^{(3)}.\end{aligned}\right.
			\end{aligned}
			$$
			Since $f_2$ satisfies (H1), we get $\left|b_s\right| \leqslant L$. Set
			$$
			X_t:=\exp \left[\int_0^t b_s d W_s-\frac{1}{2} \int_0^t\left|b_s\right|^2 d s\right] \geqslant 0 .
			$$
			Applying It\^{o}'s formula to $X_s \tilde{Y}_s$ for $s\in[t, T]$ and take conditional expectations under $\mathcal{F}_t$ on both sides, we can show that $\tilde{Y}_t$ can be expressed by the closed formula:
			$$
			\tilde{Y}_t=E^{\mathcal{F}_t}\left[\tilde{\xi}_T X_T+\int_t^T \tilde{f}_s X_s d s\right]\geqslant 0
			$$
			that is $Y_t^{(1)} \geqslant Y_t^{(3)}$, a.e., a.s.
			Then, similar to the proof of Theorem \ref{comparision}, we obtain
			$$
			Y_t^{(1)} \geqslant Y_t^{(2)}, \quad \text { a.e., a.s. }
			$$
			\\
			In the second step, we need to prove the strict comparison theorem.\\
			$(\Longrightarrow)$ Suppose $Y_0^{(1)}=Y_0^{(2)}$, then by Lemma \ref{classic strict BSDE comparison}, we get
			$$
			f_1\left(t, \{Y_r^{(1)}\}_{r \in [s,+\infty)}, Z_t^{(1)} \right)=f_2\left(t, \{Y_r^{(2)}\}_{r \in [s,+\infty)}, Z_t^{(1)}\right), \quad t \in[0, T] .
			$$
			From the first step we have that $Y_0^{(1)} \geqslant Y_0^{(3)} \geqslant Y_0^{(2)}$, thus $Y_0^{(1)}=Y_0^{(3)}$. Also by Lemma \ref{classic strict BSDE comparison}, we obtain
			$$
			f_1\left(t, \{Y_r^{(1)}\}_{r \in [s,+\infty)}, Z_t^{(1)} \right)=f_2\left(t, \{Y_r^{(1)}\}_{r \in [s,+\infty)}, Z_t^{(1)} \right), \quad t \in[0, T] .
			$$
			Therefore
			$$
			f_2\left(t, \{Y_r^{(1)}\}_{r \in [s,+\infty)}, Z_t^{(1)} \right)=f_2\left(t, \{Y_r^{(2)}\}_{r \in [s,+\infty)}, Z_t^{(1)}\right), \quad t \in[0, T] .
			$$
			Since $f_2$ is strictly increasing in the anticipation term of $Y$, we get $Y_t^{(1)}=Y_t^{(2)}, t\in[0,+\infty)$. In particular, $\xi_t^{(1)}=\xi_t^{(2)}, t\in[T,+\infty)$.
			\\
			$(\Longleftarrow)$  Suppose $f_2\left(t, \{Y_r^{(1)}\}_{r \in [s,+\infty)}, Z_t^{(1)} \right)=f_2\left(t, \{Y_r^{(2)}\}_{r \in [s,+\infty)}, Z_t^{(1)}\right), t \in[0, T]$
			and $\xi_t^{(1)}=\xi_t^{(2)},  t\in[T, +\infty)$. Then
			$$
			\tilde{Y}_t=Y_t^{(1)}-Y_t^{(3)}=E^{\mathcal{F}_t}\left[\tilde{\xi}_T X_T+\int_t^T \tilde{f}_s X_s d s\right] \equiv 0 .
			$$
			Then $Y_{.}^{(1)}$ satisfies the following IABSDE:
			$$
			\left\{\begin{aligned}
				Y_t^{(1)} = & \xi_T^{(2)}+\int_t^T f_2\left(s,\{Y_r^{(1)}\}_{r \in [s,+\infty)}, Z_s^{(3)}\right) d s-\int_t^T Z_s^{(3)} d W_s, \quad &&t\in[0, T]; \\
				Y_t^{(1)} = & \xi_t^{(2)}, \quad &&t\in[T, +\infty).
			\end{aligned}\right.
			$$
			By Theorem \ref{ex and uni}, $Y_t^{(1)}=Y_t^{(2)}$, a.e., a.s., in particular, $Y_0^{(1)}=Y_0^{(2)}$.
		\end{proof}
		%\hfill $\qedsymbol$
		\\
		\section{Stochastic Control Problem}\label{chapter Stochastic control problem}
		El Karoui, Peng and Quenez \cite{BSDEinfinance} showed a perfect duality between SDEs and classical BSDEs then applied the duality to stochastic control problems.
		In this section, we  establish a duality between SDEs with infinite delay (ISDDEs for short) and IABSDEs in Section \ref{chapter duality}, then we use the duality to solve a stochastic control problem governed by ISDDEs in Section \ref{chapter stochastic control problem}.

		\subsection{The Duality Property}\label{chapter duality}
		%that is, the solutions of linear BSDEs can expressed in terms of the solutions of corresponding linear SDEs.
		For BSDEs depend on finite anticipation, Peng and Yang \cite{peng09AP} introduced a duality between delayed SDEs and anticipated BSDEs, then Yang and Elliot \cite{yangzhe13ECP} showed that the duality also holds between generalized delayed SDEs and generalized anticipated BSDEs.
		For the BSDEs with infinite anticipation,
		%we shall prove that there exists a duality between ISDDEs and IABSDEs.
		consider the following 1-dimensional linear IABSDE:
		%, and we shall use this duality to solve a stochastic control problem in Section \ref{chapter Stochastic control problem}.
		
		\begin{equation}\label{linear IBSFDE}
			\left\{\begin{aligned}
				-dY_t & =\left(A_t\left(\{Y_r\}_{r\in [t,+\infty)}\right) + B_t\left(\{Z_r\}_{r\in [t,+\infty)}\right)+l_t\right) d t - Z_t d W_t, & & t \in[0, T]; \\
				Y_t & =Q_t, & & t \in[T, +\infty );\\
				Z_t & =P_t, & & t \in[T, +\infty ),
			\end{aligned}\right.
		\end{equation}
		where $A_t: \mathcal{M}_{\mathcal{F}}^2(t, +\infty) \rightarrow L^2\left(\mathcal{F}_t\right), B_t: \mathcal{M}_{\mathcal{F}}^{2,\beta}\left(t, +\infty ; \mathbb{R}^{1 \times m}\right) \rightarrow L^2\left(\mathcal{F}_t, \mathbb{R}^{1 \times m}\right)$ are defined by
		$$
		A_t\left(\theta\right) = E^{\mathcal{F}_t}\left[\int_{t}^{+\infty}\mu_r \theta_rd r\right], \quad
		B_t\left(\theta^{\prime}\right) = E^{\mathcal{F}_t}\left[\int_{t}^{+\infty}\nu_r \theta^{\prime}_rd r\right],
		$$
		where $\mu, \nu: \mathbb{R} \rightarrow \mathbb{R}$ are two uniformly continuous functions. Assume there exists a constant $C>0$, such that
		\begin{equation}\label{H3}
			\int_{0}^{+\infty}\left|\mu_r\right|dr \leqslant C, \quad \int_{0}^{+\infty}\nu_r^2dr \leqslant C.
		\end{equation}
		
		\begin{proposition}
			Let conditions (\ref{H3}) hold. Then, the generator in (\ref{linear IBSFDE}) satisfies assumption (H1).
		\end{proposition}
		\begin{proof}
			For all $t \in[0, T], Y_{.}, Y_{.}^{\prime} \in \mathcal{M}_\mathcal{F}^2\left(t, +\infty\right)$, it follows that
			$$
			\begin{aligned}
				\left|A_t\left(\{Y_r\}_{r\in [t,+\infty)}\right)
				-A_t\left(\{Y^{\prime}_r\}_{r\in [t,+\infty)}\right)\right|
				&\leqslant E^{\mathcal{F}_t}\left[\int_{t}^{+\infty}\left|\mu_r\right|\cdot \left(\sup_{r \leqslant k < +\infty} \left|Y_k-Y_k^{\prime} \right|\right)dr \right]\\
				&\leqslant C E^{\mathcal{F}_t} \left[\left\|Y_t-Y_t^{\prime} \right\|\right].
			\end{aligned}
			$$
			Thus assumption (H1) holds for $Y$-anticipation term. For $Z$-anticipation term, and any constant $\beta \geqslant 0$, $Z_{.}, Z_{.}^{\prime} \in \mathcal{M}_\mathcal{F}^{2,\beta}\left(t, +\infty; \mathbb{R}^{1 \times m}\right)$,
			$$
			\left|B_s\left(\{Z_r\}_{r\in [s,+\infty)}\right)
			-B_s\left(\{Z^{\prime}_r\}_{r\in [s,+\infty)}\right)\right|^2
			\leqslant \left(\int_{s}^{+\infty}\nu_r^2 e^{-\beta r}dr\right)\cdot E^{\mathcal{F}_s}\left[\int_{s}^{+\infty}\left|Z_r - Z_r^{\prime}\right|^2 e^{\beta r} dr \right].
			$$
			By Fubini's Theorem, we have
			$$
			\begin{aligned}
				E&\left[\int_{t}^{T} \left|B_s\left(\{Z_r\}_{r\in [s,+\infty)}\right)
				-B_s\left(\{Z^{\prime}_r\}_{r\in [s,+\infty)}\right)\right|^2 e^{\beta s} ds\right]\\
				&\leqslant \left(\int_{0}^{+\infty}\nu_r^2dr\right)\cdot E\left[ \int_{t}^{T} \left(\int_{s}^{+\infty}\left|Z_r - Z_r^{\prime}\right|^2 e^{\beta r} dr\right)e^{\beta s}ds\right]\\
				&\leqslant \frac{Ce^{\beta T}}{\beta}  \cdot E\left[\int_{t}^{+\infty}\left|Z_r - Z_r^{\prime}\right|^2 e^{\beta r}dr \right].
			\end{aligned}
			$$
		\end{proof}
		
		\begin{lemma} \label{duality lemma}
			Suppose $l . \in L_{\mathcal{F}}^2(0, T)$ and $\mu., \nu.$ satisfy (\ref{H3}). Then for any given terminal condition $Q. \in \mathcal{S}^2_{\mathcal{F}}\left(T, +\infty\right), P. \in \mathcal{M}_{\mathcal{F}}^{2,\beta}\left(T, +\infty; \mathbb{R}^{1 \times m}\right)$, the solution $Y.$ of the IABSDE (\ref{linear IBSFDE}) achieves the following closed formula:
			$$
			\begin{aligned}
				Y_t= & E^{\mathcal{F}_t}\left[X_T Q_T+\int_t^T X_s l_s  d s\right] \\
				& +E^{\mathcal{F}_t}\left[\left(\int_T^{+\infty}\mu_r Q_r  d r\right)\left(\int_t^T X_s d s\right)
				+ \left(\int_T^{+\infty}\nu_r P_r  d r\right)\left(\int_t^T X_s d s\right)\right] \quad \text { a.e., a.s. }
			\end{aligned}
			$$
			where $\left\{X_s\right\}_{s \in (-\infty, T]}$ is the solution of the following linear ISDDE:
			$$
			\left\{\begin{aligned}
				d X_s & =\mu_s \left(\int_{-\infty}^{s}X_rdr\right) ds+\nu_s\left(\int_{-\infty}^{s} X_rdr\right)dW_s, & & s \in[t, T]; \\
				X_t & =1, & & \\
				X_s & =0, & & s \in(-\infty, t).
			\end{aligned}\right.
			$$
		\end{lemma}
		%\noindent\textbf{Proof.}
		\begin{proof}
			%We say that the proof of Lemma \ref{duality lemma} is similar to the proof of Proposition 3.3 in \cite{yangzhe13ECP}.
			By Theorem 3.1 in \cite{07JMAA}, the above ISDDE has a unique solution $X_{.}\in \mathcal{M}_{\mathcal{F}}^2\left(-\infty, T\right)$. Applying It\^{o}'s formula to $X_s Y_s$ for $s \in[t, T]$, and taking the conditional expectation under $\mathcal{F}_t$, then by Fubini's Theorem, we have
			$$
			\begin{aligned}
				E^{\mathcal{F}_t}&\left[X_T Y_T\right]-X_t Y_t \\
				= & E^{\mathcal{F}_t}\left[-\int_t^T X_s\left(\int_s^{+\infty} \mu_r Y_r d r\right) d s-\int_t^T X_s \left(\int_s^{+\infty} \nu_r Z_r d r\right) d s-\int_t^T X_s l_s d s\right] \\
				& +E^{\mathcal{F}_t}\left[\int_t^T \mu_sY_s \left(\int_{-\infty}^s X_r d r\right) d s+\int_t^T \nu_s Z_s \left(\int_{-\infty}^s X_r d r\right) d s\right] \\
				= & E^{\mathcal{F}_t}\left[-\int_T^{+\infty} \mu_r Y_r \left(\int_t^T X_s d s\right) d r-\int_t^T \mu_r Y_r \left(\int_t^r X_s d s\right) d r\right] \\
				& +E^{\mathcal{F}_t}\left[-\int_T^{+\infty} \nu_r Z_r \left(\int_t^T X_s d s\right) d r-\int_t^T \nu_r Z_r \left(\int_t^r X_s d s\right) d r-\int_t^T X_s l_s d s\right] \\
				& +E^{\mathcal{F}_t}\left[\int_t^T \mu_sY_s \left(\int_{-\infty}^s X_r d r\right) d s
				+ \int_t^T \nu_s Z_s \left(\int_{-\infty}^s X_r d r\right) d s\right] \\
				= & E^{\mathcal{F}_t}\left[-\int_T^{+\infty} \mu_r Y_r \left(\int_t^T X_s d s\right) d r-\int_T^{+\infty} \nu_r Z_r \left(\int_t^T X_s d s\right) d r -\int_t^T X_s l_s d s \right]\\
				&+ E^{\mathcal{F}_t}\left[\int_t^T \mu_s Y_s \left(\int_{-\infty}^t X_r d r\right) d s
				+\int_t^T \nu_s Z_s \left(\int_{-\infty}^t X_r d r\right) d s\right].
			\end{aligned}
			$$
			Since $X_t=1$ and $X_s=0$ for $ s \in(-\infty, t)$, we can obtain that
			$$
			\begin{aligned}
				Y_t= & E^{\mathcal{F}_t}\left[X_T Q_T+\int_t^T X_s l_s d s\right] \\
				& +E^{\mathcal{F}_t}\left[\left(\int_T^{+\infty}\mu_r Q_r  d r\right)\left(\int_t^T X_s d s\right)
				+ \left(\int_T^{+\infty}\nu_r P_r  d r\right)\left(\int_t^T X_s d s\right)\right]. \\
			\end{aligned}
			$$
		\end{proof}
		%\hfill $\qedsymbol$
		
		\subsection{Optimal Control}\label{chapter stochastic control problem}
		%The duality between BSDEs and SDEs plays an important role in stochastic control problems. Classical results can be found in El Karoui, Peng and Quenez \cite{BSDEinfinance}, where BSDEs provide explicit representations for optimal control problems governed by SDEs. Peng and Yang \cite{peng09AP} showed a duality between anticipated BSDEs and stochastic differential delay equations (SDDEs), extending the control theory to time-delay systems, such as \cite{control19,control20,control21}.
		Now, we consider the stochastic control problem governed by the following ISDDEs with control:
		%controlled SDEs with infinite delay:
		\begin{equation}\label{ISDDE}
			\left\{\begin{aligned}
				d X_s^u & =\mu\left(s, u_s\right) \left(\int_{-\infty}^{s}X_r^udr\right) ds+\sigma\left(s,u_s\right)X_s^udW_s, & & s \in[t, T]; \\
				X_t & =1, & & \\
				X_s & =0, & & s \in(-\infty, t),
			\end{aligned}\right.
		\end{equation}
		where $\mu\left(s, u\right): \mathbb{R} \times \mathbb{R}^k \longrightarrow \mathbb{R}^{+}$and $\sigma\left(s, u\right): \mathbb{R} \times \mathbb{R}^k \longrightarrow \mathbb{R}^{1 \times m}$ are adapted processes uniformly continuous with respect to $(s, u)$.
		
		A feasible control $\left(u_s, s \in[0, +\infty)\right)$ is a continuous adapted process valued in a compact subset $U$ in $\mathbb{R}^k$, and denote the set of feasible controls by $\mathcal{U}$.
		Our objective is to maximize the following objective function over all feasible controls $u$:
		\begin{equation}\label{cost functional}
			\begin{aligned}
				J(u)= E\left[X_T^u Q(T)+\int_0^T X_s^u l\left(s, u_s\right)  d s\right]
				+E\left[\left(\int_T^{+\infty}\mu\left(s, u_s\right) Q(s)  d s\right)\left(\int_0^T X_s^u d s\right)\right],
			\end{aligned}
		\end{equation}
		where $Q(\cdot) \in \mathcal{S}_{\mathcal{F}}^2(T, +\infty)$ is the terminal condition,
		% $\left(l\left(\omega, s, u_s\right), s \in[0, T]\right)$ is the running cost and $l(s, u)$ is an adapted process uniformly continuous with respect to $(s, u)$.
		adapted process $\left(l\left(\omega, s, u_s\right), s \in[0, T]\right)$ is the running cost uniformly continuous with respect to $(s, u)$.
		
		{ In the objective function (\ref{cost functional}), we introduce a novel term that couples past actions $\left(\int_0^T X_s^u d s\right)$ with future effect $\left(\int_T^{+\infty} \mu Q d s\right)$. This structure arises naturally in non-instant transmission phenomena, e.g. climate economics or biological applications.
			\begin{remark}
				Our optimization problem (\ref{ISDDE})-(\ref{cost functional}) can be illustrated by a concrete example of carbon emission control model.
				The state process $X_s$ represents the carbon growth, and it is reasonable to assume that $X_s$ depends on the cumulative carbon emission during the past period $\left(\int_{-\infty}^{s}X_r^udr\right)$. The government implements the emission control policy $u_s$, e.g. carbon tax, to optimize the overall carbon damage in the form of (\ref{cost functional}):
				\begin{itemize}
					\item The running cost $l(s,u_s)$ captures the instantaneous production loss due to policy implementation. The terminal cost $Q(T)$ measures the immediate economic cost at time $T$ from irreversible climate impacts (e.g., sea-level rise);
					\item The coupling term $\left(\int_T^{\infty} \mu(s,u_s) Q(s)  d s\right)\left(\int_0^T X_s^u d s\right)$ models the long-term climate feedback: historical cumulative carbon emissions $\left(\int_0^T X_s^u d s\right)$ amplify future damages, scaled by the decay rate $\mu(\cdot)$ (e.g., CO$_2$ emitted today causes heating that decays slowly over centuries).
				\end{itemize}
				%where cumulative emissions [ $X_s^u$ ] amplify future damages/penalty [ $Q(s)$ ] via memory kernel $\mu(\cdot)$
			\end{remark}
			
		}
		
		Assume that $\int_{0}^{+\infty}\left|\mu(s, u)\right|^2ds$, $|\sigma(s, u)|$ and $l(s, u)$ are uniformly bounded by $C$. Further assume that $|\mu(s, u)|$ is uniformly bounded by a non-negative $h(s)$ where $\int_0^{+\infty} h(s) d s \leqslant C$, then  $\int_{t}^{+\infty}\left|\mu(s, u)\right|ds$ is also bounded by $C$ and uniformly continuous with respect to $(t, u)$.
		
		Then, by Lemma \ref{duality lemma}, $J(u)=Y_0^u$, where $\left(Y_{.}^u, Z_{.}^u\right)$ is the solution of the following linear IABSDE:
		$$
		\left\{\begin{aligned}
			Y_t^u&=Q(T) + \int_{t}^{T}f^u\left(s, \{Y_r^u\}_{r \in [s,+\infty)}, Z_s^u\right) d s - \int_{t}^{T}Z_s^u d W_s, && t \in[0, T] ; \\
			Y_t^u&=Q(t), && t \in[T, +\infty),
		\end{aligned}\right.
		$$
		where for all $t\in[0,T], Y. \in \mathcal{S}_{\mathcal{F}}^2\left(t, +\infty \right), z\in \mathbb{R}^m$
		$$
		f^u\left(t,Y.,z\right)=E^{\mathcal{F}_t}\left[\int_{t}^{+\infty}\mu\left(r,u_r\right)Y_rdr \right] + \sigma(t,u_t) z+ l\left(t, u_t\right).
		$$
		
		\begin{theorem}\label{control theorem}
			Set $f\left(t, Y_{.}, z\right)=\operatorname{esssup}_{u \in \mathcal{U}}\left\{f^u\left(t, Y_{.}, z\right)\right\}$, for all $t\in[0,T], Y. \in \mathcal{S}_{\mathcal{F}}^2\left(t, +\infty \right)$, $z\in \mathbb{R}^m$. Then the following IABSDE
			\begin{equation}\label{19}
				\left\{\begin{aligned}
					Y_t&=Q(T) + \int_{t}^{T}f\left(s, \{Y_r\}_{r \in [s,+\infty)}, Z_s\right) d s - \int_{t}^{T}Z_s d W_s, && t \in[0, T] ; \\
					Y_t&=Q(t), && t \in[T, +\infty),
				\end{aligned}\right.
			\end{equation}
			has a unique solution $\left(Y_{.},Z_{.}\right)$, and $Y_{.}$ is the value function of the optimal control problem, that is, for all $t\in[0,T]$,
			$$
			Y_t=Y_t^*:=\operatorname{esssup}_{u \in \mathcal{U}}\left\{Y_t^u\right\}.
			$$
		\end{theorem}
		%\noindent \textbf{Proof.}~~
		\begin{proof}
			In the first step, we need to prove the IABSDE (\ref{19}) has a unique solution.\\
			Since $\int_{t}^{+\infty}\left|\mu(s, u)\right|ds$ is uniformly continuous with respect to $u$, then by the Measurable Selection Theorem in Kuratowski and Ryll-Nardzewski \cite{selector}, there exists a measurable sequence $\left\{u^n\right\} \subset \mathcal{U}$ such that
			$$
			\operatorname{esssup}_{u \in \mathcal{U}} \left\{ \int_t^{+\infty}|\mu(s, u)| d s \right\} =\lim _{n \rightarrow+\infty} \int_t^{+\infty}|\mu(s, u_s^n)| d s.
			$$
			%Then, similar to the discussions in Remark \ref{H3},
			Then by Monotone Convergence Theorem,
			$$
			\begin{aligned}
				\operatorname{esssup}_{u \in \mathcal{U}} E^{\mathcal{F}_t}[\int_t^{+\infty}\left|\mu\left(s, u_s\right)\right| d s] &\leqslant \lim _{n \rightarrow+\infty} E^{\mathcal{F}_t}\left[\max_{1 \leqslant k \leqslant n} \int_t^{+\infty}\left|\mu\left(s, u_s^k\right)\right| d s\right]\\
				&= E^{\mathcal{F}_t}\left[\lim _{n \rightarrow+\infty}\int_t^{+\infty}\left|\mu\left(s, u_s^n\right)\right| d s\right]\\
				&=E^{\mathcal{F}_t}\left[\operatorname{esssup}_{u \in \mathcal{U}} \left\{ \int_t^{+\infty}|\mu(s, u)| d s \right\}\right] .
			\end{aligned}
			$$	
			Therefore, for all $t\in[0,T], Y_{.},Y_{.}^{\prime} \in \mathcal{S}_{\mathcal{F}}^2\left(t, +\infty \right)$, $z, z^{\prime}\in \mathbb{R}^m$, we have
			$$
			\begin{aligned}
				\left|f\right.& \left.\left(t, Y_{.}, z\right)- f\left(t, Y_{.}^{\prime}, z^{\prime}\right) \right|\\
				&\leqslant \operatorname{esssup}_{u \in \mathcal{U}}\left\{E^{\mathcal{F}_t}\left[\int_{t}^{+\infty}\left|\mu(s, u_s)\right|ds\right]\right\} \cdot E^{\mathcal{F}_t}\left[\left\|Y_t - Y_t^{\prime} \right\| \right]
				+ \operatorname{esssup}_{u \in \mathcal{U}}\left\{\left|\sigma(t,u_t)\right| \right\}\cdot \left|z - z^{\prime}\right|\\
				&\leqslant E^{\mathcal{F}_t}\left[\operatorname{esssup}_{u \in \mathcal{U}}\left\{\int_{t}^{+\infty}\left|\mu(s, u_s)\right|ds \right\}\right] \cdot E^{\mathcal{F}_t}\left[\left\|Y_t - Y_t^{\prime} \right\| \right]
				+ \operatorname{esssup}_{u \in \mathcal{U}}\left\{\left|\sigma(t,u_t)\right|\right\}\cdot \left|z - z^{\prime}\right|\\
				&\leqslant C\left(E^{\mathcal{F}_t}\left[\left\|Y_t - Y_t^{\prime} \right\| \right] + \left|z - z^{\prime}\right|\right),
			\end{aligned}
			$$
			and
			$$
			E\left[\int_0^T|f(t, 0,0)|^2 d t\right] \leqslant C^2 T.
			$$
			Thus, $f$ satisfies assumptions (H1) and (H2), then by Theorem \ref{ex and uni}, the IABSDE (\ref{19}) has a unique solution $(Y., Z.) \in \mathcal{S}_\mathcal{F}^2\left(0, +\infty\right) \times \mathcal{M}_\mathcal{F}^{2}\left(0, T; \mathbb{R}^{1 \times m}\right)$.\\

			In the second step, we need to prove that $Y_t=Y_t^*$. \\On the one hand, $f(t, Y_{.}, z) \geqslant f^u(t, Y_{.}, z)$ for all $u \in \mathcal{U}$,
			and $ f^u(t, Y_{.}, z)$ is increasing in the $Y$-anticipation term, then by Comparison Theorem \ref{comparision}, we have $Y_t \geqslant Y_t^u$, a.e., a.s. Thus,
			$$Y_t \geqslant Y_t^*, \quad \text { a.e., a.s. }$$
			On the other hand, for all $\varepsilon \textgreater 0$, define the set-valued function $F^{\varepsilon}:[0,T] \times \Omega \rightarrow 2^{U}$:
			$$
			F^{\varepsilon}\left(t, \omega \right)
			= \left\{u \in \mathcal{U}: f\left(t,\left\{Y_r(\omega)\right\}_{r\in[t,+\infty)}, Z_t(\omega)\right) \leqslant f^u\left(t,\left\{Y_r(\omega)\right\}_{r\in[t,+\infty)}, Z_t(\omega)\right) + \varepsilon    \right\}.
			$$
			By the definition of $f$, for all $(t, \omega) \in [0, T) \times \Omega$, $F^{\varepsilon}\left(t, \omega \right)$ is a non-empty closed set. Therefore, also by the Measurable Selection Theorem in Kuratowski and Ryll-Nardzewski \cite{selector}, there exists a $u^{\varepsilon} \in \mathcal{U}$ such that
			$$
			f\left(t,\left\{Y_r\right\}_{r\in[t,+\infty)}, Z_t\right) \leqslant f^{u^{\varepsilon}}\left(t,\left\{Y_r\right\}_{r\in[t,+\infty)}, Z_t\right) + \varepsilon, \quad \text { a.e., a.s. }
			$$
			We denote the solution of the IABSDE corresponding to $\left(f^{u^{\varepsilon}}, Q(\cdot)\right)$ by $\left(Y_{.}^{u^{\varepsilon}}, Z_{.}^{u^{\varepsilon}}\right)$, and denote $\left(\Delta Y_{.}, \Delta Z_{.}\right)=\left(Y_{.}^{u^{\varepsilon}}-Y_{.}, Z_{.}^{u^{\varepsilon}}-Z_{.}\right)$.  Note that $\left(\Delta Y_{.}, \Delta Z_{.}\right)$ satisfies
			$$
			\begin{aligned}
				\Delta Y_t&=\int_t^{T}\left(f^{u^{\varepsilon}}\left(s, \left\{Y_r^{u^{\varepsilon}}\right\}_{r\in[s,+\infty)}, Z_s^{u^{\varepsilon}}\right) - f^{u^{\varepsilon}}\left(s, \left\{Y_r\right\}_{r\in[s,+\infty)}, Z_s\right)+R(s)\right) d s-\int_t^{T} \Delta Z_s d W_s \\
				&=\int_t^{T}\left(E^{\mathcal{F}_s}\left[\int_s^{+\infty} \mu\left(r, u_r^{\varepsilon}\right) \Delta Y_r d r\right]+\sigma\left(s, u_s^{\varepsilon}\right) \Delta Z_s+R(s)\right) d s-\int_t^{T} \Delta Z_s d W_s,
			\end{aligned}
			$$
			where $R(s)=f^{u^{\varepsilon}}\left(s, \left\{Y_r\right\}_{r\in[s,+\infty)}, Z_s\right) - f\left(s, \left\{Y_r\right\}_{r\in[s,+\infty)}, Z_s\right)$ and $\left|R(s)\right|\leqslant \varepsilon$. Applying It\^{o}'s formula to $\left|\Delta Y_s\right|^2$ for $s \in[t, T]$, and taking expectation on both sides, we have
			$$
			\begin{aligned}
				E\left[\left|\Delta Y_t\right|^2 + \int_{t}^{T} \left|\Delta Z_s\right|^2 ds\right]
				\leqslant& 2	E\left[\int_{t}^{T}\left|\Delta Y_s\cdot E^{\mathcal{F}_s}\left[\int_s^{+\infty} \mu\left(r, u_r^{\varepsilon}\right) \Delta Y_r d r\right]\right|ds\right]\\
				&+2	E\left[\int_{t}^{T}\left|\sigma\left(s, u_s^{\varepsilon}\right)\Delta Y_s \Delta Z_s\right|+\left|R(s)\Delta Y_s \right| ds\right]\\
				\leqslant&E\left[\int_{t}^{T}\left|\Delta Y_s\right|^2 +  \left(\int_{s}^{+\infty}|\mu\left(r, u_r^{\varepsilon}\right)|^2dr\right)\left(\int_{s}^{T}\left|\Delta Y_r\right|^2dr\right) ds   \right]\\
				&+E\left[\int_{t}^{T}\left|\sigma\left(s, u_s^{\varepsilon}\right)\Delta Y_s\right|^2 +\left|\Delta Z_s\right|^2 +\left|\Delta Y_s\right|^2 +\left|R(s)\right|^2ds  \right].
			\end{aligned}
			$$
			Then by Fubini's Theorem, one obtains
			$$
			E\left[\left|\Delta Y_t\right|^2\right] \leqslant (CT+C^2+2)	E\left[\int_{t}^{T} \left|\Delta Y_s\right|^2 ds \right] + T\varepsilon^2.
			$$
			By Gronwall's inequality, it follows that
			$$
			E\left[\left|\Delta Y_t\right|^2\right] \leqslant Te^{CT+C^2+2}\varepsilon^2:=\rho^2\varepsilon^2.
			$$
			Since $Y_t^{u^{\varepsilon}} \leqslant Y_t$, a.e., a.s., then
			$$
			Y_t^{u^{\varepsilon}}-Y_t \geqslant -\rho \varepsilon, \quad \text { a.e., a.s. }
			$$
			Letting $\varepsilon \rightarrow 0$, we get $Y_t^{u^{\varepsilon}} \rightarrow Y_t$, a.e., a.s. Thus,
			$$
			Y_t=Y_t^*, \quad \text { a.e., a.s. }
			$$
		\end{proof}
		
		\section{Conclusion}

		In this study, we introduce an Infinite Anticipation Backward Stochastic Differential Equations (IABSDEs), where the  generator depends on the entire future path of the solution.
		
		Based on a weaker Lipschitz condition of the generator, we establish the well-posedness of IABSDEs with a Picard iteration scheme. We also obtain the comparison theorems for 1-dimensional IABSDEs under monotonicity generator, extending the results of finite anticipation case. Furthermore, we investigate a duality relation between linear IABSDEs and ISDDEs, and solve a stochastic optimization problem governed by ISDDE which demonstrates our framework can be applied to address intrinsically long-memory phenomena.
		
		A couple of directions remain to be explored in the future: Apply the duality in Section \ref{chapter duality} to establish Pontryagin-type maximum principles for IABSDE-controlled systems; Developing efficient numerical schemes e.g., Euler discretizations or deep BSDE solvers \cite{futurework-computation}, to overcome the curse of dimensionality inherent in infinite anticipation.

		\bibliography{citation}
		
	\end{document}